\documentclass[twoside,final,a4paper,10pt]{amsart}

\newcommand{\R}{\mathds R}
\newcommand{\C}{\mathds C}
\newcommand{\Z}{\mathds Z}

\newcommand{\Dim}{\dim}
\newcommand{\Ker}{{\rm Ker}\,}

\renewcommand{\ker}{\Ker}

\newcommand{\Bsym}{\mathrm B_{\mathrm{sym}}}
\newcommand{\Fredsa}{\mathcal F_{\mathrm{sa}}}
\newcommand{\spfl}{\mathfrak{sf}}

\newcommand{\Codim}{\mathrm{codim}}
\newcommand{\Imm}{\mathrm{Im}\,}
\newcommand{\Gr}{\mathrm{Graph}}

\newcommand{\FP}{\mathcal{FP}}

\newcommand{\PR}[2]{P_{\scriptscriptstyle{#1}}^{\scriptscriptstyle{#2}}}

\newcommand{\iMaslov}{\mathrm i_{\scriptscriptstyle{\textrm{Maslov}}}}

\newcommand{\ind}{\mathrm{ind}}
\newcommand{\Ddt}{\tfrac{\mathrm D}{\mathrm dt}}

\newcommand{\Ecal}{\mathcal E}
\newcommand{\Comm}{\mathcal C}

\newcommand{\dist}{\mathrm{dist}}

\newcommand{\subsetneq}{\,\substack{\subset\\ \scriptscriptstyle{\neq}}\,}

\newcommand{\codim}{\mathop\mathrm{codim}\nolimits}
\renewcommand{\Im}{\mathop\mathrm{Im}\nolimits}
\newcommand{\pp}{^\perp}
\newcommand{\ppb}{^{\perp_B}}
\newcommand{\vpb}{\mathcal V\ppb}
\newcommand{\wt}{\widetilde T}
\newcommand{\wh}{\widehat T}
\newcommand{\ra}{\langle}
\newcommand{\la}{\rangle}

\newcommand{\bex}{\begin{example}\label}
\newcommand{\eex}{\end{example}}
\newcommand{\ie}{\begin{ex}}
\newcommand{\fe}{\end{ex}}
\newcommand{\rrbx}{\begin{rex}}
\newcommand{\rrex}{\end{rex}}
\newcommand{\bde}{\begin{defin}\label}
\newcommand{\ede}{\end{defin}}
\newcommand{\bpr}{\begin{prop}\label}
\newcommand{\epr}{\end{prop}}
\newcommand{\ble}{\begin{lem}\label}
\newcommand{\ele}{\end{lem}}
\newcommand{\bco}{\begin{cor}\label}
\newcommand{\eco}{\end{cor}}
\newcommand{\bre}{\begin{rem}\label}
\newcommand{\ere}{\end{rem}}
\newcommand{\beq}{\begin{equation}\label}
\newcommand{\eeq}{\end{equation}}
\newcommand{\ben}{\begin{enumerate}}
\newcommand{\een}{\end{enumerate}}
\newcommand{\bit}{\begin{itemize}}
\newcommand{\eit}{\end{itemize}}
\newcommand{\bth}{\begin{teo}\label}
\newcommand{\eth}{\end{teo}}

\renewcommand{\contentsline}[3]{\csname new#1\endcsname{#2}{#3}}
\newcommand{\newchapter}[2]{\bigskip\hbox to \hsize{\vbox{\advance\hsize by -.5cm\baselineskip=12pt\parfillskip=0pt\leftskip=2cm\noindent\hskip -2cm #1\leaders\hbox{.}\hfil\hfil\par}$\,$#2\hfil}}
\newcommand{\newsection}[2]{\medskip\hbox to \hsize{\vbox{\advance\hsize by -.5cm\baselineskip=12pt\parfillskip=0pt\leftskip=2.5cm\noindent\hskip -2cm #1\leaders\hbox{.}\hfil\hfil\par}$\,$#2\hfil}}
\newcommand{\newsubsection}[2]{\medskip\hbox to \hsize{\vbox{\advance\hsize by -.5cm\baselineskip=12pt\parfillskip=0pt\leftskip=3.5cm\noindent\hskip -2cm #1\leaders\hbox{.}\hfil\hfil\par}$\,$#2\hfil}}

\usepackage{times}
\usepackage{dsfont}
\usepackage[all]{xy}

\numberwithin{equation}{section}

\title[On a formula for the spectral flow]{On a formula for the spectral flow and its applications}
\author[P. Benevieri]{Pierluigi Benevieri}
\address{Dipartimento di Matematica Applicata,
Universit\`a degli Studi di Firenze, \hfill\break\indent Via S. Marta 3, 50139 Firenze, Italy}
\email{pierluigi.benevieri@unifi.it}

\author[P.\ Piccione]{Paolo Piccione}
\address{Departamento de Matem\'atica,
Universidade de S\~ao Paulo, \hfill\break\indent Rua do Mat\~ao
1010, CEP 05508-900, S\~ao Paulo, SP, Brazil}
\email{piccione@ime.usp.br}
\urladdr{http://www.ime.usp.br/\~{}piccione}
\thanks{P. P. is partially sponsored by CNPq and Fapesp.}

\date{December 6th, 2007}

\begin{document}


\theoremstyle{plain}\newtheorem*{teon}{Theorem}
\theoremstyle{definition}\newtheorem*{defin*}{Definition}
\theoremstyle{plain}\newtheorem{teo}{Theorem}[section]
\theoremstyle{plain}\newtheorem{prop}[teo]{Proposition}
\theoremstyle{plain}\newtheorem{lem}[teo]{Lemma}
\theoremstyle{plain}\newtheorem{cor}[teo]{Corollary}
\theoremstyle{definition}\newtheorem{defin}[teo]{Definition}
\theoremstyle{remark}\newtheorem{rem}[teo]{Remark}
\theoremstyle{plain} \newtheorem{assum}[teo]{Assumption}
\swapnumbers
\theoremstyle{definition}\newtheorem{example}{Example}[section]
\theoremstyle{plain} \newtheorem*{acknowledgement}{Acknowledgements}
\theoremstyle{definition}\newtheorem*{notation}{Notation}


\begin{abstract}
We consider a continuous path of bounded symmetric Fredholm bilinear forms
with arbitrary endpoints on a real Hilbert space, and we prove a formula that
gives the spectral flow of the path in terms of the spectral flow of
the restriction to a finite codimensional closed subspace.
We also discuss the case of restrictions to a continuous path of finite
codimensional closed subspaces.
As an application of the formula, we introduce the notion of spectral flow
for a periodic semi-Riemannian geodesic, and we compute its value in terms
of the Maslov index.
\end{abstract}

\maketitle

\tableofcontents


\begin{section}{Introduction}

The notion of \emph{spectral flow} plays a central role in several areas of Calculus of Variations,
including Morse theory and bifurcation theory; this is a fixed endpoint homotopy invariant integer associated to continuous
paths of Fredholm symmetric bilinear forms on Hilbert spaces.
In the modern formulations of Morse theory, it is now well understood that this notion
is the natural substitute for the notion of Morse index for critical points of strongly indefinite
variational problems. For instance, under suitable assumptions, the dimension of the intersection
of stable and unstable manifolds of critical points of a smooth functional $f$ defined on a Hilbert manifold
is given by the  spectral flow of the Hessian of $f$ along the flow lines of $\nabla f$ joining the two
critical points (see \cite{AbbMej2}). In bifurcation theory, jumps of the spectral flow detect
bifurcation from some given branch of critical points of a smooth curve of strongly indefinite smooth functionals
(see \cite{FitzPejsaRecht}). Starting from the celebrated work of T. Yoshida \cite{Y}, a series of
results have been proven in the literature relating the spectral flow of a path of
Dirac operators on partitioned manifolds to the geometry of the Cauchy data spaces
(see \cite{BB, BF2, BF1, BLP, BooWoj, BW,  CLM3, Nico}); low dimensional topological invariants can
be computed in terms of spectral flow (see \cite{CLM3,Y}).

A natural question in the above problems is to compute the spectral flow of \emph{restrictions}
to a given closed subspace, or more generally to a continuous path of closed subspaces,
of a continuous path of Fredholm bilinear forms. In Calculus of Variations, restriction
of the Hessian of smooth functionals corresponds to studying constrained variational problems.
For instance, the typical Fredholm forms arising from geometrical variational problems
are obtained from self-adjoint differential operators acting on sections of vector bundles
over (compact) manifolds with boundary satisfying suitable boundary conditions.
A formula for the spectral flow of restrictions in this case would allow to reduce
the study of a general boundary condition to the usually easier case of
Dirichlet conditions.

The aim of this paper is to prove formulas (Theorem~\ref{riduzionespettro}
and Proposition~\ref{thm:propsubs}) relating the spectral flow of a continuous
path of Fredholm symmetric bilinear forms to the spectral flow of their restriction
to a continuous path of finite codimensional closed subspaces, which is still Fredholm (Lemma~\ref{thm:restrFred}).

Let us recall that the spectral flow of a path of symmetric bilinear forms is given by an algebraic count of
eigenvalues passing through $0$
in the spectrum of the path of self-adjoint operators that represent the bilinear forms relatively to some
choice of inner products. However, a spectral theoretical approach to the restriction problem would not be
successful, due to the fact that restrictions of bilinear forms correspond to left multiplication
by a projection, and this operation in general perturbs the spectrum of a self-adjoint operator
in a quite complicated way. In order to prove the desired result, we will use a different characterization
of the spectral flow, which is given in terms of relative dimension of Fredholm pairs in the
Grassmannian of all closed subspaces of a Hilbert space. The spectral flow of a path of
Fredholm self-adjoint operators of the form symmetry plus compact is given by the relative
dimension of the negative spectral subspaces at the endpoints. One proves that a finite codimensional
reduction does not destroy the symmetry plus compact form of a Fredholm operator (Lemma~\ref{symmetryrestrictions});
moreover, the relative dimension of the negative eigenspaces behaves well with respect to
compact perturbations (Proposition~\ref{final}).

The case of restrictions to a varying family
of closed finite codimensional subspaces (Proposition~\ref{thm:propsubs}) is reduced to the case of a fixed subspace by means of
a special class of trivialization of the family.
We observe that one does not lose generality in considering only the case of paths of the form symmetry plus compact.
Namely, let us recall that the spectral flow is invariant by the cogredient action of the general linear
group of the Hilbert space on the space of self-adjoint Fredholm operators, and that all the orbits of
this action meet the affine space of compact perturbations of a fixed symmetry. By an elementary
principal fiber bundle argument, every path of class $C^k$, $k=0,\ldots,\infty,\omega$, in the space of
self-adjoint Fredholm operators is cogredient to a $C^k$ path of compact perturbations of a symmetry.

The paper is finalized with the discussion of an application of our reduction formula
in the context of semi-Riemannian geometry (Section~\ref{sec:applicationsRGeom}). We will consider an
orientation preserving periodic geodesic $\gamma$ in
a semi-Riemannian manifold $(M,g)$, and we will define its \emph{spectral flow}, as a suitable
generalization of the Morse index of the geodesic action functional, defined on the free loop space of $M$,
at the critical point $\gamma$. Observe that, unless
the metric tensor $g$ is positive definite, the standard Morse index of every nontrivial
closed geodesic is infinite. Unlike the fixed endpoint case, in the periodic case the
definition of spectral flow depends heavily on the choice of a periodic frame along
the geodesic. Two distinct choices of a periodic frame along a given closed geodesic produce
two paths of self-adjoint Fredholm operators that are in general neither fixed endpoint homotopic
nor cogredient. Recall in analogy that periodic solutions of Hamiltonian systems on general symplectic
manifolds do \emph{not} have a well defined Conley--Zehnder index (i.e., independent of the choice of a periodic
symplectic frame along the solution), unless one poses serious
restrictions on the topology of the underlying manifold.

An application of Theorem~\ref{riduzionespettro}  gives us
a formula for the spectral flow of a periodic geodesic (Theorem~\ref{Morseperiodic}), given
in terms of the \emph{Maslov index}  and the so-called \emph{concavity index} of the geodesic,
plus a certain degeneracy term. The Maslov index is a symplectic invariant which is associated
to the underlying fixed endpoint geodesic, while the concavity index is an integer invariant
of periodic solutions of Hamiltonian systems, which was introduced by M. Morse in the context
of Riemannian closed geodesics. A first, and somewhat surprising, consequence of the formula,
is that the spectral flow is well defined regardless of the choice of a periodic frame.
This fact is probably more interesting \emph{in se} than the formula itself.
Further developments of the theory are to be expected in the realm of Morse theory
for semi-Riemannian periodic geodesics, which at the present stage is a largely unexplored field
(see \cite{BMP} for the stationary Lorentzian case, or \cite{AbbMej} for the fixed endpoints Lorentzian case).
A natural conjecture would be that, under suitable nondegeneracy assumptions, the difference
of spectral flows at two distinct geodesics is equal to the dimension of the intersection
between the stable and the unstable manifolds of the gradient flow at the two critical points
in the free loop space.

An effort has been made in order to make the paper essentially self-contained.
In Section~\ref{sec:preliminaries} we recall a few preliminary basic facts on Fredholm
operators and bilinear forms; the central result is Proposition~\ref{isomorse}, that gives
an upper bound for the dimension of an isotropic subspace.
Section~\ref{sec:pairs} contains most of the basic facts in the theory
of Fredholm pairs and commensurable pairs of closed subspaces and relative dimension, with complete proofs.
The main result (Proposition~\ref{final}) is a formula giving the relative dimension of the negative eigenspaces of a self-adjoint
Fredholm operator and its restriction to any closed finite codimensional subspace of a Hilbert space.
Section~\ref{sec:spflow} contains material on the spectral flow, dealing mostly with the case of
paths of Fredholm operators that are compact perturbations of a fixed symmetry.
Theorem~\ref{riduzionespettro} gives a formula for the computation of the spectral flow of a path
of Fredholm symmetric bilinear forms (with arbitrary endpoints) in terms of the spectral flow of
its restriction to a finite codimensional closed subspaces, and some boundary terms. Observe that
both the path and/or its restriction is allowed to admit degeneracies at the endpoints.
In Proposition~\ref{thm:propsubs} we show how the  same result can be employed to study the case of restrictions to a continuous
path of closed finite codimensional subspaces. A discussion of the notion of continuity, or smoothness,
for a path of closed subspaces of a Hilbert space is presented in subsection~\ref{sub:contclosedsubspaces}.
Smoothness for a path is defined in terms of the smoothness of local trivializations for the
path (Definition~\ref{thm:defsmoothnessclosedsubspaces}); we show that this is equivalent to the smoothness of the corresponding path of orthogonal
projections in the Banach algebra of all bounded operators on the Hilbert spaces (Proposition~\ref{thm:smoothnessequivalent}).
This characterization of continuity yields several interesting facts. First, as it is shown
in Appendix~\ref{app:grassmannians}, one can find \emph{global} trivializations, second,
the trivialization can be chosen by a path of isometries of the Hilbert space.
Such trivialization will be called an orthogonal trivialization; orthogonal trivializations are special
cases of the so-called \emph{splitting} trivializations, that are employed in the
definition of spectral flow in the case of restriction to varying domains.
Section~\ref{sec:applicationsRGeom} contains the geometrical application of the theory.

\end{section}

\begin{section}{Preliminaries. Fredholm bilinear forms.}
\label{sec:preliminaries}


In this section we will recall some basic facts about the geometry of closed subspaces of a
Hilbert space, and, in addition, some properties of bounded symmetric bilinear forms
on Hilbert spaces. Basic references for this part are \cite[Chapter~2]{Ka},
\cite[Section~2]{BooWoj}, \cite[Section~1]{Nico} and  \cite[Chapter~4, \S~4]{Ab}.
Virtually, most of the material discussed is
well known to specialists; the authors' intention
is merely to fix notations and to state the results in
a way which is best suited for the purposes of the paper.

Throughout this paper we will denote by $\mathcal H$ a real separable Hilbert space, endowed
with inner product $\langle\cdot,\cdot\rangle$; by $\Vert\cdot\Vert$ we will indicate the relative norm.
Many of the results presented here will {\em not\/} indeed
depend on the choice of a specific Hilbert space
inner product. Complex extensions of the theory are also very likely to exist, but we
will not be concerned with the complex case here.

Given a closed subspace $\mathcal V$ of $\mathcal H$, $P_\mathcal V$ will stand for the orthogonal projection onto $\mathcal V$, and $\mathcal V\pp$ will denote the orthogonal complement of $\mathcal V$ in $\mathcal H$. Depending on the context we will use the same symbol $P_\mathcal V$ for the projection with target space $\mathcal H$ or $\mathcal V$.
Given two closed subspaces $\mathcal V$ and
 $\mathcal W$  of $\mathcal H$, $\PR{ \mathcal V}{\mathcal W}$ will represent the restriction to $\mathcal W$ of $P_\mathcal V$; an immediate calculation shows that the adjoint of $ \PR{ \mathcal V}{\mathcal W}$ is $\PR{ \mathcal W}{\mathcal V}$.

Let us warm up by singling out a few basic facts concerning closed subspaces,
orthogonal projections and compact operators, that will be used explicitly or
implicitly in our proofs.

\begin{lem}\label{thm:fatti}
Let $\mathcal V$ and $\mathcal W$  be closed subspaces of $\mathcal H$; the following statements hold true:
\begin{enumerate}
\item  \label{itm:nucleosomma}
$\Ker(P_\mathcal V+P_\mathcal W)=\mathcal V^\perp\cap \mathcal W^\perp$, and $\overline{\Imm(P_\mathcal V+P_\mathcal W)}=
\big(\Ker(P_\mathcal V+P_\mathcal W)\big)^\perp=\overline{\mathcal V+\mathcal W}$;
\item\label{itm:codimfinita} if $\Codim(\mathcal V+\mathcal W)<+\infty$, then $\mathcal V+\mathcal W$ is closed;
\item\label{itm:immI+K} if $K:\mathcal H\to\mathcal H$ is a compact linear operator, then
$(I+K)\mathcal V$ is closed;
\item\label{itm:codimsub} if $\mathcal V\supseteq \mathcal W^\perp$, then $\Codim  \mathcal V=\Codim_\mathcal W(\mathcal V\cap \mathcal W)$;
\item\label{itm:contienecodimfinita} if $\Codim\mathcal V<+\infty$, then any subspace of $\mathcal H$ containing $\mathcal V$  is closed;
\item\label{itm:dimfinita} If $\Dim\mathcal V<+\infty$, then $\Dim\big((\mathcal V+\mathcal W^\perp)\cap \mathcal W\big)<+\infty$.
\end{enumerate}
\end{lem}

\begin{proof}
To prove \eqref{itm:nucleosomma} observe in first place that
$\Ker(P_\mathcal V+P_\mathcal W)\supseteq \mathcal V^\perp\cap \mathcal W^\perp$. If $x\in \Ker(P_\mathcal V+P_\mathcal W)$, then
\[
\Vert P_\mathcal V x\Vert^2=\langle P_\mathcal V x,x\rangle=-\langle P_\mathcal Wx,x\rangle=-
\Vert P_\mathcal Wx\Vert^2,
\]
hence $\Vert P_\mathcal Vx\Vert=\Vert P_\mathcal Wx\Vert=0$, and $x\in \mathcal V^\perp\cap \mathcal W^\perp$.
The second equality in \eqref{itm:nucleosomma} follows immediately.

Statement \eqref{itm:codimfinita} follows from the general fact
that, given a bounded linear operator between Banach spaces
$T:\mathcal F\to \mathcal G$, having image of finite codimension, then $\Imm T$ is closed.
This is an easy application of the Open Mapping Theorem.
In the case, $\mathcal V+\mathcal W$ is the image of the bounded operator from
 $\mathcal V\times \mathcal W$ to $\mathcal H$, given by $(x,y)\mapsto x+y$.

The proof of \eqref{itm:contienecodimfinita} goes as follows. Let $\mathcal U$ be any subspace of $\mathcal H$ containing
$ \mathcal V$, and consider the
quotient map $\pi:\mathcal H\to \mathcal H/\mathcal V$. Since this quotient is finite
dimensional, then $\pi(\mathcal U)$ is closed, and, since $\mathcal U\supseteq \mathcal V=\Ker\pi$, then
  $\mathcal U$ is saturated, i.e., $\mathcal U=\pi^{-1}(\pi(\mathcal U))$, which implies that
$\mathcal U$ is closed.

To prove \eqref{itm:dimfinita} consider $\PR{ \mathcal W}{\mathcal V}:\mathcal V\to \mathcal W$, which clearly has finite dimensional, hence closed, image. Then
\[
\begin{array}{l}
\Imm\PR{ \mathcal W}{\mathcal V}=\overline{\Imm\PR{ \mathcal W}{\mathcal V}}=\big(\Ker(
\PR{ \mathcal W}{\mathcal V})^*\big)^\perp=(\Ker\PR{ \mathcal V}{\mathcal W})^\perp=\smallskip\\
(\mathcal V^\perp\cap \mathcal W)^\perp
\cap \mathcal W=(\mathcal V+ \mathcal W^\perp)\cap \mathcal W.
\end{array}
\]
This concludes the proof.
\end{proof}

Moreover, an application of \cite[Ch.~4, \S~4, Theorem~4.2]{Ka} yields the following Lemma.
\begin{lem}
\label{thm:sommachiusa}
Let $\mathcal V,\mathcal W$ be closed subspaces of $\mathcal H$.
Then $\mathcal V+\mathcal W$ is closed if and only if the operator $\PR{ \mathcal V^\perp}{\mathcal W}:\mathcal W\to \mathcal V^\perp$
has closed image.
\end{lem}

\begin{proof}
For $w\in \mathcal W$, set $w_o=P_{\mathcal V\cap \mathcal W}w$ and $w_\perp=w-w_o$;
the result of \cite[Ch.~4, \S~4, Theorem~4.2]{Ka} tells us that
$\mathcal V+\mathcal W$ is closed if and only if there exists $c>0$ such that
$\Vert P_{\mathcal V^\perp}w_\perp\Vert\ge c\Vert w_\perp\Vert$ for all
$w\in \mathcal W$.
In turn, this latter condition is equivalent to the
fact that  $P_{\mathcal V^\perp}\vert_\mathcal W:\mathcal W\to \mathcal V^\perp$ has closed image.
\end{proof}

\begin{rem}\label{thm:remgamma}
Given closed subspaces $\mathcal V,\mathcal W$ of a Banach space, let us recall Kato's
definition of the constant $\gamma(\mathcal V,\mathcal W)\in\left[0,1\right]$:
\[
\gamma(\mathcal V,\mathcal W)=\inf_{u\in \mathcal V\setminus \mathcal W}\frac{\dist(u,\mathcal W)}{\dist(u,\mathcal V\cap \mathcal W)}.
\]

It is proven in \cite[Ch.~4, \S~4, Theorem~4.2]{Ka} that $\mathcal V+\mathcal W$ is
closed if and only if $\gamma(\mathcal V,\mathcal W)>0$. Similarly, if $L$ is a bounded linear operator between Banach spaces,
then the image of $L$ is closed if and only if the constant
\[
\gamma(L)=\inf_{u\not\in\Ker L}\frac{\Vert Lu\Vert}{\dist\big(u,\Ker L\big)}
\]
is positive. An immediate calculation shows that if $\mathcal V,\mathcal W$ are closed subspaces
of a Hilbert space $\mathcal H$, then $\gamma(\mathcal V,\mathcal W)=\gamma\big(\PR{\mathcal V^\perp}{\mathcal W}\big)$;
from this fact it follows immediately a proof of Lemma~\ref{thm:sommachiusa}.
\end{rem}

Let now $B$ be a continuous bilinear form on $\mathcal H$ and $T:\mathcal H\to \mathcal H$ the continuous linear operator uniquely associated with $B$, that is,
\[
B(x,y)=\ra Tx, y\la, \quad \forall x,y\in \mathcal H.
\]

We define
\[
\ker B=\{x\in\mathcal H: B(x,y)=0, \; \forall y\in \mathcal H\}.
\]
It is immediate to see that $\ker B=\ker T$. If $\ker B=\{0\}$, then $B$ is said to be \emph{nondegenerate}.

If a continuous bilinear form $B$ is symmetric, then $T$ is \emph{self-adjoint}, that is, $\ra Tx,y\la = \ra x,Ty \la $, for all $x,y\in \mathcal H$.

\bde{morsedef}
Given a continuous bilinear form $B$, the \emph{Morse index} of $B$ is the (possibly infinite) integer number
\[
\mathrm n_-(B)= \sup\big\{\dim \mathcal V: \; B|_{\mathcal V\times \mathcal V} \text{ is negative definite}\big\}.
\]
We denote by $\mathrm n_+(B)$ the Morse index of $-B$, also called the \emph{Morse coindex} of $B$. Of course one has
\[
\mathrm n_+(B)= \sup\big\{\dim \mathcal V: \; B|_{\mathcal V\times \mathcal V} \text{ is positive definite}\big\}.
\]
\ede

\bde{fredholmform} A symmetric continuous bilinear form $B$ on $\mathcal H$, associated with a (self-adjoint) Fredholm operator, is called a \emph{symmetric Fredholm form} on $\mathcal H$.
\ede
A self-adjoint Fredholm operator has null index.

\noindent\textbf{Standing assumption.} From now on $B$ will denote a symmetric Fredholm form on $\mathcal H$
and $T$ will be the self-adjoint Fredholm operator $T$ associated with $B$.

\medskip

By the spectral theory of the self-adjoint Fredholm operators, there exists a unique orthogonal splitting of $\mathcal H$ induced by $B$,
\beq{decspet}
H=V^-(T)\oplus V^+(T)\oplus \ker T,
\eeq
such that $V^-(T)$ and  $V^+(T)$ are both $T$-invariant,  $B|_{V^-(T)\times V^-(T)}$ is negative definite and $B|_{V^+(T)\times V^+(T)}$ is positive definite.

In addition, since $V^-(T)$ and  $V^+(T)$ are $T$-invariant and orthogonal, they are also $B$-\emph{orthogonal}, that is, $B(x,y)=0$ for any $x\in V^-(T)$ and any  $y\in V^+(T)$.

With a slight abuse of notation, we will refer to $V^-(T)$ and  $V^+(T)$ respectively as the \emph{negative} and the \emph{positive eigenspaces} of $B$.

\bre{morsenegative}
Observe that the Morse index of a symmetric Fredholm form $B$ coincides with the (possibly infinite) dimension of the negative eigenspace $V^-(T)$.
\ere

Given a subspace $\mathcal V$ of $\mathcal H$, we define the $B$-\emph{orthogonal complement} of $\mathcal V$ as the subspace of $\mathcal H$
\[
\mathcal V\ppb=\{x\in \mathcal H: B(x,y)=0,\; \forall y\in \mathcal V\}.
\]

\bre{osssulortog} Given a closed subspace $\mathcal V$ of $\mathcal H$, we have the following properties.
\begin{itemize}
\item[i)] $\vpb$ is closed and $\ker T\subseteq \vpb$, the proof is immediate.
\item[ii)] If $\mathcal V$ has finite codimension, then $\vpb$ is finite dimensional. Indeed,
\[
\vpb=\{x\in \mathcal H:\ra Tx,y\la=0,\; \forall y\in \mathcal V\}=\{x\in \mathcal H:\ra x,Ty\la=0,\; \forall y\in \mathcal V\}.
\]
That is, $\vpb$ is orthogonal to $T({\mathcal V})$, which has finite codimension since $T$ is Fredholm and $\mathcal V$ has finite codimension. More precisely,
\[
\dim \vpb=\codim \mathcal V +\dim \big(\ker T\cap{\mathcal V}\big).
\]
\item[iii)] Analogously, if $\mathcal V$ has finite dimension, then $\vpb$ has finite codimension coinciding with $\dim \mathcal V-\dim \ker T|_\mathcal V$.
\item[iv)] In general, $\mathcal V+\vpb\neq\mathcal  H$, even when $B$ is  nondegenerate.
\end{itemize}
\ere

\ble{thm:restrFred}
If $\mathcal V$ is a closed subspace of $\mathcal H$, having finite
codimension, then the restriction $B\vert_{\mathcal V\times\mathcal V}$ is Fredholm.
\ele

\begin{proof}
The kernel of $B\vert_{\mathcal V\times\mathcal V}$ is given by $\mathcal V\cap\mathcal V^{\perp_B}$, which is finite dimensional.
If $T$ is the Fredholm self-adjoint operator that represents $B$, then $B\vert_{\mathcal V\times\mathcal V}$ is represented by $P_{\mathcal V}\circ T\vert_{\mathcal V}$,
whose image contains $T(\mathcal V)\cap\mathcal V$, which has finite codimension.
\end{proof}

Let $\mathcal V$ be a closed subspace of $\mathcal H$. Denote by $\wt:\mathcal V\to \mathcal V$ the operator associated  with $B|_{\mathcal V\times \mathcal V}$ and by $T_2:\vpb\to \vpb$  the operator associated with  $B|_{\vpb\times \vpb}$. Notice that $T_2=P_{\vpb}\circ T|_{\vpb}$.

\ble{bortsommadir} In the above notation we have the following results.
\begin{enumerate}
\item  If $\mathcal V\cap \vpb=\{0\}$, then $B|_{\mathcal V\times \mathcal V}$ is nondegenerate.
\item If $\mathcal V$ is finite dimensional or finite codimensional and $B|_{\mathcal V\times \mathcal V}$ is nondegenerate, then $\mathcal H=\mathcal V\oplus \vpb$.
\item $\ker \wt$ and $\ker T$ are contained in $\ker T_2$. If in particular $\mathcal H=\mathcal V+\vpb$ (not necessarily direct sum), then $\ker T=\ker T_2$.
\item If $B$ is nondegenerate and $\mathcal V+\vpb=\mathcal H$, then $\mathcal V\cap \vpb=\{0\}$.
\item If $\mathcal V$ is finite dimensional or finite codimensional, then
\[
(\mathcal V\cap \vpb)\ppb=\mathcal V+\vpb.
\]
\end{enumerate}
\ele

\begin{proof}
(1) If $x\in \ker \wt$, then $Tx$ is orthogonal to $\mathcal V$, that is, $x\in \vpb$. Hence $x=0$ and $B|_{\mathcal V\times \mathcal V}$ is nondegenerate.

\medskip

(2) Let $v\in \mathcal V\cap \vpb$ be given. As $v \in \vpb$, it is orthogonal to $T(\mathcal V)$, that is, $0=\ra Tv',v \la= \ra v',Tv \la$ for any $v'\in \mathcal V$. This implies that $Tv$ is orthogonal to $\mathcal V$
and so $\wt v=0$ ($v$ belongs to $\mathcal V$, hence $\wt v$ is well defined). Thus $v=0$ since $B|_{\mathcal V\times \mathcal V}$ is nondegenerate. Notice that the proof that $\mathcal V\cap \vpb=\{0\}$ does not require any information about the dimension of $\mathcal V$.

Now, if $\mathcal V$ has finite codimension, then $\vpb$ has finite dimension. Hence, if $\mathcal V$ is finite dimensional or finite codimensional, then $\mathcal V+\vpb$ is closed being the sum of two closed subspaces of $\mathcal H$ such that one of them has finite dimension.

To show that $\mathcal V+\vpb =\mathcal H$ consider an element $v$ of the orthogonal complement of $\mathcal V+\vpb$ in $\mathcal H$. We have that $v\in T(\mathcal V)$ since this latter coincides with $(\vpb)\pp$. Let $x\in \mathcal V$ be such that $Tx=v$. As $Tx$ is orthogonal to $\mathcal V$, then $\wt x=0$ and this implies that $x=0$ since $\wt$ is injective. Therefore, $v=0$ and we have finally
$\mathcal H=\mathcal V\oplus\vpb$.

\medskip

(3) If $x\in \ker T$, then $\ra Tx,y \la=0$ for each $y\in \mathcal V$, that is, $x\in \vpb$ and $T_2 x$ is well defined. As $Tx=0$, trivially $T_2x=0$, that is, $\ker T\subseteq \ker T_2$. Given $x\in \ker \wt$, in the decomposition $\mathcal H=\mathcal V\oplus \mathcal V\pp$ write $Tx=\wt x+P_{\mathcal V\pp}Tx=P_{\mathcal V\pp}Tx$. Hence $\ra Tx,y \la=0$ for each $y\in \mathcal V$, that is, $x\in \vpb$ and $T_2 x$ is well defined. In the decomposition $\mathcal H = (\vpb)\pp\oplus \vpb $ denote by $Q$ the orthogonal projection onto $(\vpb)\pp$. Then,
\[
0=\ra Tx,y \la = \ra QTx+T_2x,y \la = \ra T_2x,y \la, \quad \forall y\in \vpb.
\]
Then $T_2x=0$, that is, $\ker \wt\subseteq \ker T_2$.

In the particular case when $\mathcal H=\mathcal V+\vpb$, let $x\in \ker T_2$ be given. Given any $z\in \mathcal H$, let us write $z=z_1+z_2$, where $z_1\in \mathcal V$ and $z_2\in \vpb$. We have
\[
\ra Tx,z\la=\ra Tx,z_1\la + \ra Tx,z_2\la.
\]
The product $\ra Tx,z_1\la$ vanishes since $x\in \vpb$ and $z_1\in \mathcal V$, and the term $\ra Tx,z_2\la$ is zero since $T_2x=0$, that is $Tx\in (\vpb)\pp$. Hence $\ra Tx,z\la=0$ for any $z\in \mathcal H$, that means $Tx=0$.

\medskip

(4) Let $x\in\mathcal  V\cap \vpb$ be given. Given any $z\in \mathcal H$,  write $z=z_1+z_2$, where $z_1\in \mathcal V$ and $z_2\in \vpb$. Then
\[
\ra Tx,z \la = \ra Tx,z_1 \la + \ra Tx, z_2 \la=0.
\]
In fact, $\ra Tx,z_1 \la=0$ since $x\in \vpb$ and $z_1\in \mathcal V$, while $\ra Tx,z_2 \la=0$ since $x\in \mathcal V$ and $z_2\in \vpb$. Hence $\ra Tx,z \la =0$ for any $z\in\mathcal  H$ and this implies that $x=0$ since $B$ is nondegenerate.

\medskip

(5) It is a consequence of the following properties shown (in a more general setting) in \cite{BMP}: \emph{given two closed subspaces $S_1$ and $S_2$ of $\mathcal H$, then}
\begin{itemize}
\item [i)] $(S_1+S_2)\ppb=S_1\ppb\cap S_2\ppb$,
\item [ii)] $(S_1\ppb)\ppb=S_1+\ker T$.
\end{itemize}

First of all one can show that
\[
(\mathcal V\cap \vpb)\ppb=((\mathcal V+\ker T)\cap \vpb)\ppb
\]
(even if $\mathcal V\cap \vpb$ could be strictly contained in $(\mathcal V+\ker T)\cap \vpb$; this is the case when $\mathcal V$ does not contain $\ker T$). Indeed, fix an element $x\in (\mathcal V\cap \vpb)\ppb $ and let $w\in (\mathcal V+\ker T)\cap \vpb$ be given. One can write $w=v+k$, where $v\in \mathcal V$ and $k\in\ker T$. Since $\ker T\subseteq\vpb$, then $k$ belongs to $\vpb$, as $w$, and thus $v\in \vpb $ as well. That is, $v\in \mathcal V\cap \vpb$ and this implies
\[
\ra Tx,w \la=\ra Tx,v \la+\ra Tx,k \la=0,
\]
since $\ra Tx,v \la$ and $\ra Tx,k \la$ both vanish. Therefore, $x\in ((\mathcal V+\ker T)\cap \vpb)\ppb$, that is, $(\mathcal V\cap \vpb)\ppb\subseteq((\mathcal V+\ker T)\cap \vpb)\ppb$.

The inclusion $((\mathcal V+\ker T)\cap \vpb)\ppb\subseteq(\mathcal V\cap \vpb)\ppb$ follows immediately from the inclusion $\mathcal V\cap \vpb\subseteq(\mathcal V+\ker T)\cap \vpb$.

Let us now conclude the proof of the statement (5). By the previous item ii) we have $\mathcal V+\ker T=(\vpb)\ppb$. Hence, by i), $(\mathcal V+\ker T)\cap \vpb=(\vpb+\mathcal V)\ppb$. By ii), $((\vpb+\mathcal V)\ppb)\ppb=\vpb+\mathcal V$, recalling that $\ker T\subseteq \vpb$.

Summarizing the arguments,
\[
(\mathcal V\cap \vpb)\ppb=((\mathcal V+\ker T)\cap \vpb)\ppb=\vpb+\mathcal V
\]
and the proof is complete.
\end{proof}


\bre{strictly}
If $\mathcal V+\vpb$ is strictly contained in $\mathcal H$, $\ker T$ does not necessarily coincides with $\ker T_2$ and $\mathcal  V\cap \vpb$ is not necessarily empty. Examples, even in finite dimension, could be easily provided and left to the reader.
\ere

\bde{isotropic}
A subspace $\mathcal Z$ of $\mathcal H$ is said to be \emph{isotropic} for $B$ if $B(z,z)=0$ for any $z\in \mathcal Z$.
\ede

Any subspace of $\ker T$ is clearly isotropic, but one can easily find examples of symmetric Fredholm forms having isotropic subspaces not contained in the kernel of the associated operator.

\ble{isononkernel}
Suppose that $B$ admits an isotropic subspace $\mathcal Z$ which is not contained in $\ker T$. Then $B$ is indefinite, that is, there exist $x,y\in \mathcal H$ such that $\ra Tx,x \la>0$ and $\ra Ty,y \la<0$.
\ele

\begin{proof}
Let $v\in \mathcal H$ such that $\ra Tv,v \la=0$ and $w:=Tv\neq 0$. For any $\alpha \in \R$ we have
\[
\ra T(\alpha v+w),\alpha v+w \la=2 \alpha ||w||^2+\ra Tw,w \la.
\]
The claim follows  choosing $x=\alpha_1 v+w$ and $y=\alpha_2 v+w$, with any $\alpha_1>-\dfrac{|\ra Tw,w \la|}{2||w||^2}$ and any $\alpha_2<-\dfrac{|\ra Tw,w \la|}{2||w||^2}$.
\end{proof}

\begin{cor}\label{ortpos}
If $B$ is positive (resp. negative) semidefinite, then $B|_{(\ker T)\pp\times (\ker T)\pp}$ is positive (resp. negative) definite.
\end{cor}

Lemma \ref{isononkernel} above  allows us to prove the next result connecting the Morse index of $B$ and a given isotropic space $\mathcal Z$.

\bpr{isomorse}
If $\mathcal Z$ is an isotropic subspace of $\mathcal H$, then
\[
\dim \mathcal Z \leq  \mathrm n_-(B)+ \dim (\mathcal Z\cap \ker T) \quad \text{and} \quad \dim \mathcal Z \leq \mathrm n_+(B)+ \dim (\mathcal Z\cap \ker T).
\]
\epr

\begin{proof}
Let us prove just the first inequality, the proof of the second one is analogous.
If $\mathcal Z$ is infinite dimensional (this is the case when, for instance, it is not closed), one has $\mathrm n_-(B)=+\infty$ and this could be easily verified using the proof of the above Lemma \ref{isononkernel}. In this case the inequality $\dim \mathcal Z \leq  \mathrm n_-(B)+ \dim (\mathcal Z\cap \ker T)$ immediately follows.

Suppose now that $\dim \mathcal Z<+\infty$. If $\mathcal Z$ is contained in $\ker T$, the result trivially holds. If $\mathcal Z$ is not contained in $\ker T$, then $B$ is indefinite and $ \mathrm n_-(B)$ is strictly positive (or $+\infty$).
Call $\mathcal V$ the orthogonal complement of $\mathcal Z\cap \ker T$ in $\mathcal Z$ and  recall the spectral decomposition \eqref{decspet} of $\mathcal H$, induced by $B$, $\mathcal H=V^-(T)\oplus V^+(T)\oplus \ker T$.

Given $z\in \mathcal V$, if $P_{V^- (T)} z=0$, then $z\in V^+(T)$ and thus
\[
\ra Tz,z \la \geq 0.
\]

On the other hand $\ra Tz,z \la = 0$ as $z$ belongs to $\mathcal Z$. Then $z=0$ and $\PR{V^- (T)}{\mathcal V}$ is injective. Consequently
\[
\dim \mathcal Z -\dim (\mathcal Z\cap \ker T) = \dim \mathcal V =\dim (\Im\PR{V^- (T)}{\mathcal V})\leq\mathrm n_-(B)
\]
and the proposition is proven.
\end{proof}
\end{section}


\begin{section}{Fredholm and commensurable pairs of closed subspaces}
\label{sec:pairs}
\subsection{Relative dimension and Fredholm pairs}

The following notion of Fredholm pair of closed subspaces of $\mathcal H$ has been introduced by Kato (see \cite{Ka}).

\begin{defin}
Given two closed subspaces $\mathcal V$ and $\mathcal  W$ of $\mathcal H$, we will say that $(\mathcal V,\mathcal W)$
is a {\em Fredholm pair\/} if $\dim(\mathcal V\cap \mathcal W)<+\infty$ and $\Codim(\mathcal V+\mathcal W)<+\infty$.
We will denote by $\FP(\mathcal H)$ the set of all Fredholm pairs
of closed subspaces in $\mathcal H$; for $(\mathcal V,\mathcal W)\in\FP(\mathcal H)$ we set
\[
\ind(\mathcal V,\mathcal W)=\Dim(\mathcal V\cap \mathcal W)-\Codim(\mathcal V+\mathcal W).
\]
\end{defin}

We observe that, by part~\eqref{itm:codimfinita} of Lemma~\ref{thm:fatti},
if $(\mathcal V,\mathcal W)\in\FP(\mathcal H)$ then $\mathcal V+\mathcal W$ is closed, and so
\[
\ind(\mathcal V,\mathcal W)=\Dim(\mathcal V\cap \mathcal W)-\Dim\big((\mathcal V+\mathcal W)^\perp\big)=\Dim(\mathcal V\cap \mathcal W)-\Dim(\mathcal V^\perp\cap \mathcal W^\perp).
\]

Establishing if a given pair of closed subspaces is a Fredholm pair
is not always easy; usually,  the nontrivial part of the proof is
to show that the sum of the spaces is closed. Once this is done,
the finite codimensionality is obtained using orthogonality arguments.
For this reason, it will be essential to determine criteria
of Fredholmness of pairs; most of such criteria are given
in terms of orthogonal projections.

\begin{prop}\label{thm:coppieoperatori} Given two closed subspaces $\mathcal V$ and $\mathcal  W$ of $\mathcal H$,
$(\mathcal V,\mathcal W)\in\FP(\mathcal H)$ if and only if
$\PR{\mathcal V^\perp}{\mathcal W}:\mathcal W\to \mathcal V^\perp$ is a Fredholm operator.
In this case, $\ind(\mathcal V,\mathcal W)$ equals the Fredholm index of $\PR{\mathcal V^\perp}{\mathcal W}$.
\end{prop}

\begin{proof}
In first place,
\begin{equation}\label{eq:coppieoperatori1}
\Ker \PR{\mathcal V^\perp}{\mathcal W}=\mathcal V\cap \mathcal W.
\end{equation}
If $( \mathcal V, \mathcal W)\in\FP(\mathcal H)$, then $\mathcal V+\mathcal W$ is closed, and, by Lemma~\ref{thm:sommachiusa},
$\PR{\mathcal V^\perp}{\mathcal W}$ has closed image.
Moreover,
\begin{multline}
\label{eq:coppieoperatori2}
\Imm\PR{\mathcal V^\perp}{\mathcal W}=
\overline{\Imm\PR{\mathcal V^\perp}{\mathcal W}}=\big(\Ker(\PR{\mathcal V^\perp}{\mathcal W})^*\big)^\perp=(\Ker\PR {\mathcal W}{\mathcal V^\perp})^\perp\\=
(\mathcal V^\perp\cap \mathcal W^\perp)^\perp\cap \mathcal V^\perp=\overline{\mathcal V+\mathcal W}\cap \mathcal V^\perp=(\mathcal V+\mathcal W)\cap \mathcal V^\perp,
\end{multline}
(see also the proof of part \eqref{itm:dimfinita} in Lemma \ref{thm:fatti}). By part~\eqref{itm:codimsub} of Lemma~\ref{thm:fatti},
\begin{equation}\label{eq:coppieoperatori3}
\Codim_{\mathcal V^\perp}((\mathcal V+\mathcal W)\cap \mathcal V^\perp)=\Codim(\mathcal V+\mathcal W)<+\infty,
\end{equation}
hence
$\PR{\mathcal V^\perp}{\mathcal W}$ is Fredholm.
Conversely, if $\PR{\mathcal V^\perp}{\mathcal W}$ is Fredholm, then, by Lemma~\ref{thm:sommachiusa}, $\mathcal V+\mathcal W$ is closed;
moreover, the adjoint $(\PR{\mathcal V^\perp}{\mathcal W})^*=\PR {\mathcal W}{\mathcal V^\perp}$ is also Fredholm, and thus
\[
\Codim(\mathcal V+\mathcal W)=\Dim\big((\mathcal V+\mathcal W)^\perp\big)=\Dim(\mathcal V^\perp\cap \mathcal W^\perp)=
\Dim(\Ker\PR {\mathcal W}{\mathcal V^\perp})<+\infty.
\]

The last statement in the thesis follows readily from \eqref{eq:coppieoperatori1}, \eqref{eq:coppieoperatori2}
and \eqref{eq:coppieoperatori3}.
\end{proof}

\begin{cor}\label{thm:propFP}
If $(\mathcal V,\mathcal W)\in\FP(\mathcal H)$, then $(\mathcal W,\mathcal V)$ and $(\mathcal V^\perp,\mathcal W^\perp)$ are in
$\FP(\mathcal H)$, and $\ind(\mathcal V,\mathcal W)=\ind(\mathcal W,\mathcal V)=-\ind(\mathcal V^\perp,\mathcal W^\perp)$.
\end{cor}

\begin{proof}
The fact that $(\mathcal W,\mathcal V)\in\FP(\mathcal H)$ follows directly from the definition
of Fredholm pairs, as well as the equality $\ind(\mathcal V,\mathcal W)=\ind(\mathcal W,\mathcal V)$.
Moreover, since the adjoint of $\PR{\mathcal V^\perp}{\mathcal W}$ is $\PR {\mathcal W}{\mathcal V^\perp}$, it
follows from Proposition~\ref{thm:coppieoperatori} that $(\mathcal V^\perp,\mathcal W^\perp)\in\FP(\mathcal H)$,
and that $\ind(\mathcal V^\perp,\mathcal W^\perp)=\ind(\PR {\mathcal W}{\mathcal V^\perp})=-\ind(\PR{\mathcal V^\perp}{\mathcal W})=
-\ind(\mathcal V,\mathcal W)$.
\end{proof}

Here is yet another characterization of Fredholm pairs.

\begin{cor}\label{thm>diffprojFred}
$(\mathcal V,\mathcal W)\in\FP(\mathcal H)$ if and only if the difference $P_\mathcal V-P_\mathcal W:\mathcal H\to\mathcal H$
is Fredholm.
\end{cor}

\begin{proof}
Consider the operators
\[
\begin{array}{ll}
\wt:\mathcal H\to \mathcal W\oplus \mathcal W^\perp, \quad & \wt(x)=(P_\mathcal Wx, P_{\mathcal W^\perp}x),\smallskip\\
T_2:\mathcal W\oplus \mathcal W^\perp\to \mathcal V^\perp\oplus \mathcal V, \quad & T_2(w,w_\perp)=(P_{\mathcal V^\perp}w,-P_\mathcal Vw_\perp),\smallskip\\
T_3:\mathcal V^\perp\oplus \mathcal V\to \mathcal H,\quad & T_3(v,v_\perp)=v+v_\perp.
\end{array}
\]

Clearly,
$\wt$ and $T_3$ are isomorphisms, and the composition $T_3\circ T_2\circ \wt:\mathcal H\to\mathcal H$
is Fredholm if and only if $T_2$ is Fredholm. We have
\begin{multline*}
T_3(T_2(\wt z))=T_3(T_2(P_\mathcal Wz,P_{\mathcal W^\perp}z))=P_{\mathcal V^\perp}(P_\mathcal Wz)-
P_\mathcal V(P_{\mathcal W^\perp}z)\\=P_{\mathcal V^\perp}(P_\mathcal Wz)+P_\mathcal V(P_\mathcal Wz)-P_\mathcal V(P_\mathcal Wz)-
P_\mathcal V(P_{\mathcal W^\perp}z)=P_\mathcal Wz-P_\mathcal Vz.
\end{multline*}
Now, $T_2=\PR{\mathcal V^\perp}{\mathcal W}\oplus(-\PR{\mathcal V}{\mathcal W^\perp})$,
and this is Fredholm if
and only if both $\PR{\mathcal V^\perp}{\mathcal W}$ and $\PR{\mathcal V}{\mathcal W^\perp}$ are; the conclusion
follows now from Proposition~\ref{thm:coppieoperatori} and Corollary~\ref{thm:propFP}.
\end{proof}

As to the \emph{sum} of orthogonal projections onto Fredholm pairs,
we have the following result.

\begin{lem}\label{thm:sommasuriettiva}
Let $\mathcal V,\mathcal W$ be closed subspaces of $\mathcal H$ such that $\mathcal V\cap \mathcal W=\{0\}$ and such
that $\mathcal V+\mathcal W$ is closed.
Then, the image of  $P_\mathcal V+P_\mathcal W:\mathcal H\to\mathcal H$ is $\mathcal V+\mathcal W$. In particular, if
$\mathcal V+\mathcal W=\mathcal H$, then $P_\mathcal V+P_\mathcal W$ is surjective.
\end{lem}

\begin{proof}
Obviously, $\mathrm{Im}(P_\mathcal V+P_\mathcal W)\subseteq \mathcal V+\mathcal W$.
Since $\mathcal V+\mathcal W$ is closed and $P_\mathcal VP_{\mathcal V+\mathcal W}=P_\mathcal V$, $P_\mathcal WP_{\mathcal V+\mathcal W}=P_\mathcal W$,
we can replace $\mathcal H$ by $\mathcal V+\mathcal W$ and assume that
$\mathcal V+\mathcal W=\mathcal H$.

Since $\mathcal V\cap \mathcal W=\{0\}$, then there exists a (unique) linear operator $A:\mathcal V^\perp\to \mathcal V$
whose graph
\[
\Gr(A)=\big\{z+Az:z\in \mathcal V^\perp\big\}\subseteq \mathcal H
\]
is $\mathcal W$.
Then, $\mathcal H=\mathcal V+\Gr(A)$.

Clearly, $A$ is bounded because its graph is closed
(Closed Graph Theorem). It is easy to show that the graph of the negative
adjoint map $-A^*:\mathcal V\to \mathcal V^\perp$ is equal to $\mathcal W^\perp$; namely, if  $y\in \mathcal W$, then $y=z+Az$ for some $z\in \mathcal V^\perp$. Now, if $x\in \mathcal V$, we have
\[
\langle x-A^*x,y\rangle=\langle x-A^*x,z+Az\rangle=-\langle A^*x,z\rangle+\langle x,Az\rangle=0,
\]
i.e., $\Gr(-A^*)\subseteq \mathcal W^\perp$. On the other hand, choose $t\in \mathcal W^\perp$
and write $t=t_\mathcal V+t_{\mathcal V^\perp}$, where $t_\mathcal V\in \mathcal V$ and $t_{\mathcal V^\perp}\in \mathcal V^\perp$.
Since $\mathcal W=\Gr(A)$, we have:
\[
\langle t,z+Az\rangle=0,\quad\forall z\in \mathcal V^\perp,
\]
i.e.,
\[
0=\langle t_\mathcal V+t_{\mathcal V^\perp},z+Az\rangle=\langle t_\mathcal V,Az\rangle+\langle t_{\mathcal V^\perp},z\rangle=
\langle A^*t_\mathcal V+t_{\mathcal V^\perp},z\rangle
\]
for all $z\in \mathcal V^\perp$, which implies $A^*t_\mathcal V+t_{\mathcal V^\perp}\in \mathcal V$.
But $A^*t_\mathcal V+t_{\mathcal V^\perp}\in \mathcal V^\perp$, hence $A^*t_\mathcal V+t_{\mathcal V^\perp}=0$, and
$t_{\mathcal V^\perp}=-A^*t_\mathcal V$, $t=t_\mathcal V-A^*t_\mathcal V\in\Gr(-A^*)$, that is
$\Gr(-A^*)\supseteq \mathcal W^\perp$. That is, $\Gr(-A^*)= \mathcal W^\perp$.

Let us now determine the image of $P_\mathcal V+P_\mathcal W$;
let $r\in \mathcal H$ be fixed,
we search $s\in\mathcal H$ with $P_\mathcal Vs+P_\mathcal Ws=r$. Write $r=z+Az+t$, with $z\in \mathcal V^\perp$ and
$t\in \mathcal V$, and  set
\[
s=(z+Az)+(c-A^*c),
\]
where $c=t-Az\in \mathcal V$.
Observe that $z+Az\in \mathcal W$ and $c-A^*c\in \mathcal W^\perp$, i.e., $P_\mathcal Ws=z+Az$. Writing $s=(Az+c)+(z-A^*c)$, we have
 $Az+c\in \mathcal V$ and $z-A^*c\in \mathcal V^\perp$. Hence $P_\mathcal Vs=Az+c=t$.
In conclusion, $P_\mathcal Vs+P_\mathcal Ws=z+Az+t=r$ and the proof is concluded.
\end{proof}

As to the image of $P_\mathcal V+P_\mathcal W$ for a general Fredholm pair $(\mathcal V,\mathcal W)$, we have the following lemma.

\begin{lem}\label{thm:imagesum}
Given a Fredholm pair $(\mathcal V,\mathcal W)$, the image of $P_\mathcal V+P_\mathcal W$ has finite codimension in $\mathcal H$.
\end{lem}
\begin{proof}
By Proposition~\ref{thm:coppieoperatori}, $\PR{ \mathcal V^\perp}{\mathcal W}:\mathcal W\to \mathcal V^\perp$ and $\PR{ \mathcal W^\perp}{\mathcal V}  :\mathcal V\to \mathcal W^\perp$
are Fredholm, and thus their adjoints $\PR{ \mathcal W}{\mathcal V^\perp}   :\mathcal V^\perp\to \mathcal W$ and $\PR{ \mathcal V}{\mathcal W^\perp}   :\mathcal W^\perp\to \mathcal V$
are Fredholm. It follows that $\mathcal W'=P_\mathcal W(\mathcal V^\perp)$ has finite codimension in $\mathcal W$, and that
$\mathcal V'=P_\mathcal V(\mathcal W^\perp)$ has finite codimension in $\mathcal V$. But $(P_\mathcal V+P_\mathcal W)(\mathcal W^\perp+\mathcal V^\perp)=
P_\mathcal V(\mathcal W^\perp)+P_\mathcal W(\mathcal V^\perp)=\mathcal V'+\mathcal W'$, hence the image of $P_\mathcal V+P_\mathcal W$ has finite codimension in $\mathcal V+\mathcal W$. Since
$\mathcal V+\mathcal W$ has finite codimension in $\mathcal H$, it follows that $P_\mathcal V+P_\mathcal W$ has image of finite codimension in
$\mathcal H$.
\end{proof}

We can now extend the result of Lemma~\ref{thm:sommasuriettiva}
to pairs of closed subspaces $\mathcal V$ and $\mathcal W$ whose intersection is not
zero.

\begin{prop}\label{thm:sommafredcoppiafred}
Let $\mathcal V,\mathcal W\subseteq\mathcal H$ be closed subspaces with $\mathrm{dim}(\mathcal V\cap \mathcal W)<+\infty$.
Then, $P_\mathcal V+P_\mathcal W$ is Fredholm if and only if $(\mathcal V,\mathcal W)$ is a Fredholm pair.
\end{prop}

\begin{proof}
If $P_\mathcal V+P_\mathcal W$ is Fredholm, then $\mathcal V+\mathcal W$ is a closed and finite codimensional subspace of $\mathcal H$
because it contains the image of $P_\mathcal V+P_\mathcal W$.
Conversely, if $(\mathcal V,\mathcal W)$ is a Fredholm pair, by part~\eqref{itm:nucleosomma} of Lemma~\ref{thm:fatti}
one has $\Ker(P_\mathcal V+P_\mathcal W)=\mathcal V^\perp\cap \mathcal W^\perp=(\mathcal V+\mathcal W)^\perp$, hence $\mathrm{dim}\big(\Ker(P_\mathcal V+P_\mathcal W)\big)<+\infty$.
By Lemma~\ref{thm:imagesum}, the image of $P_\mathcal V+P_\mathcal W$ has finite codimension in $\mathcal H$, which concludes the proof.
\end{proof}

Set:
\[\Ecal(\mathcal H)=\Big\{(\mathcal V,\mathcal W):(\mathcal V,\mathcal W^\perp)\in\FP(\mathcal H)\Big\}.
\]
It follows immediately from Proposition~\ref{thm:coppieoperatori}
that $(\mathcal V,\mathcal W)\in\Ecal(\mathcal H)$ if and only if $\PR{\mathcal V}{^\mathcal W}$ is Fredholm.

\begin{cor}\label{thm:equivrel}
$\Ecal(\mathcal H)$ is an equivalence relation in the set of all closed
subspaces of $\mathcal H$. If $(\mathcal V,\mathcal W),(\mathcal W,\mathcal Z)\in\Ecal(\mathcal H)$, then $\ind(\mathcal V,\mathcal Z^\perp)=\ind(\mathcal V,\mathcal W^\perp)+\ind(\mathcal W,\mathcal Z^\perp)$.
\end{cor}

\begin{proof}
The reflexivity and the symmetry of $\Ecal(\mathcal H)$ follow  easily
from Corollary~\ref{thm:propFP}.
The transitivity and equality on the index will follow by proving that $\PR {\mathcal V^\perp}{\mathcal Z^\perp}$ is a
compact (in fact, a finite rank) perturbation of the composition $\PR{\mathcal V^\perp}{\mathcal W^\perp}\circ \PR{\mathcal W^\perp}{\mathcal Z^\perp}$,
using the fact that the Fredholm index of operators is stable by compact perturbations,
and additive by composition. Consider the difference
$P_{\mathcal V^\perp}-P_{\mathcal V^\perp}P_{\mathcal W^\perp}=
P_{\mathcal V^\perp}P_{\mathcal W}$; we have
\[
\Ker(P_{\mathcal V^\perp}P_{\mathcal W})=P_\mathcal W^{-1}(\mathcal V)=\overline{\mathcal V+\mathcal W^\perp}=\mathcal V+\mathcal W^\perp.
\]
Hence
\[
\Ker\big(P_{\mathcal V^\perp}P_{\mathcal W}\vert_{\mathcal Z^\perp}\big)=(\mathcal V+\mathcal W^\perp)\cap \mathcal Z^\perp.
\]
Such a space has finite codimension in $\mathcal Z^\perp$, because
\[
\big((\mathcal V+\mathcal W^\perp)\cap \mathcal Z^\perp
\big)^\perp\cap \mathcal Z^\perp=\overline{(\mathcal V^\perp\cap \mathcal W)+\mathcal Z}\cap \mathcal Z^\perp=
\big((\mathcal V^\perp\cap \mathcal W)+\mathcal Z\big)\cap \mathcal Z^\perp.\]
The last equality follows from the fact that $\mathcal V^\perp\cap \mathcal W$ is finite dimensional, so that
$(\mathcal V^\perp\cap\mathcal W)+\mathcal Z$ is closed; moreover, $\big((\mathcal V^\perp\cap\mathcal W)+\mathcal Z\big)\cap\mathcal Z^\perp$
 has finite dimension
(recall part~\eqref{itm:dimfinita} of Lemma~\ref{thm:fatti}).
This shows that the restriction of $P_{\mathcal V^\perp}-P_{\mathcal V^\perp}P_{\mathcal W^\perp}$ to
$\mathcal Z^\perp$ has finite rank, which concludes the proof.
\end{proof}


\subsection{Commensurable subspaces}
Let us now recall the notion of commensurable spaces and relative dimension, introduced in \cite{Ab0} (see also \cite{Ab}).
\begin{defin}
\label{commspaces}
Two closed subspaces $\mathcal V$ and $\mathcal W$ of $\mathcal H$ are called \emph{commensurable} if $P_\mathcal V-P_\mathcal W$ is a compact operator. The \emph{relative dimension} of $\mathcal V$ with respect to $\mathcal W$ is defined as
\[
\dim(\mathcal V, \mathcal W)=\dim \mathcal V\cap \mathcal W\pp-\dim  \mathcal W\cap \mathcal V\pp.
\]
\end{defin}

An easy computation shows that $P_\mathcal V-P_\mathcal  W$ is compact if and only if so are both $P_{\mathcal V\pp}P_\mathcal  W$ and $P_{ \mathcal W\pp}P_\mathcal V$. Indeed:
\[P_\mathcal V-P_\mathcal W=P_\mathcal V(P_\mathcal W+P_{\mathcal W^\perp})-(P_\mathcal V+P_{\mathcal V^\perp})P_\mathcal W=P_\mathcal VP_{\mathcal W^\perp}-P_{\mathcal V^\perp}P_\mathcal W,\]
and
\[P_\mathcal VP_{\mathcal W^\perp}=(P_\mathcal V-P_\mathcal W)P_{\mathcal W^\perp},\quad P_{\mathcal V^\perp}P_\mathcal W=P_{\mathcal V^\perp}(P_\mathcal W-P_\mathcal V).\]

As a consequence, if $\mathcal V$ and $ \mathcal W$ are commensurable, then $I-P_{\mathcal V\pp} P_ \mathcal W$ and $I-P_{\mathcal  W\pp}P_\mathcal V$ are Fredholm operators of index zero being compact perturbations of Fredholm operators of index zero ($I$ denotes the identity on $\mathcal H$). Therefore,
\[
 \mathcal W\cap \mathcal V\pp=\ker (I-P_{\mathcal V\pp} P_ \mathcal W) \quad  \text{and} \quad \mathcal V\cap \mathcal  W\pp =\ker (I-P_{\mathcal  W\pp}P_\mathcal V)
\]
are finite dimensional and then the above definition of relative dimension is well posed.

If follows directly from the definition that commensurability
is an equivalence relation in the set of closed subspaces
of $\mathcal H$; we will set
\[\Comm(\mathcal H)=\Big\{(\mathcal V,\mathcal W):\text{ $\mathcal V$ is commensurable
with $\mathcal W$}\Big\}.\]

Let us see the following property (see \cite{Ab}).
\begin{lem}
\label{catena}
If $\mathcal V$, $\mathcal  W$ and $\mathcal Z$ are closed commensurable subspaces of $\mathcal H$, then
\[
\dim(\mathcal V,\mathcal Z)=\dim(\mathcal V, \mathcal W)+\dim(\mathcal  W,\mathcal Z).
\]
\end{lem}

\bre{codimfin}
Two subspaces $\mathcal V$ and $\mathcal  W$ of $\mathcal H$ of finite codimension are commensurable since $P_{\mathcal  W\pp} P_\mathcal V$ and $P_{\mathcal V\pp} P_\mathcal  W$ are compact having finite dimensional image. If $\codim \mathcal V=n$ and $\codim \mathcal  W=m$, by the above lemma it follows
\[
\dim(\mathcal V,\mathcal  W)=\dim(\mathcal V,\mathcal H)+\dim(\mathcal H,\mathcal  W)=m-n.	
\]

In particular, if $L:\mathcal H\to \mathcal H$ is a Fredholm operator of index zero, then $(\ker L)\pp$ e $\Im L$ are commensurable and their relative dimension is zero.

This property clearly fails  if $\mathcal V$ or $\mathcal  W$ has infinite codimension. Consider also the particular case when
$\mathcal H=\mathcal H_1\oplus \mathcal H_2$, direct sum of infinite dimensional subspaces, and
\[
L=\left(
\begin{array}{cc}
0 & L_{12}\\
L_{21} & 0\\
\end{array}
\right)
\]
where $L_{12}$ e $L_{21}$ are isomorphisms. Then $\mathcal H_1$ e $\mathcal H_2$ are isomorphic, but not commensurable.
\ere

\begin{prop}\label{thm:commimplicaFred}
$\Comm(\mathcal H)\subsetneq\Ecal(\mathcal H)$. If $(\mathcal V,\mathcal W)\in\Comm(\mathcal H)$, then
\beq{commfred}
\dim(\mathcal V,\mathcal  W)=\ind(\mathcal V,\mathcal  W\pp).
\eeq

\end{prop}
\begin{proof}
If $(\mathcal V,\mathcal W)\in\Comm(\mathcal H)$, then the difference $P_\mathcal V-P_\mathcal W$ is compact,
and so the kernel of the Fredholm operator $I+P_\mathcal V-P_\mathcal W$ is finite dimensional:
\[
\Ker(I+P_\mathcal V-P_\mathcal W)=\Ker(P_\mathcal V+P_{\mathcal W^\perp})=\mathcal V^\perp\cap \mathcal W.
\]
On the other hand,
\[
\Codim(\mathcal V^\perp+\mathcal W)\le
\Codim(\Imm(P_{\mathcal V^\perp}+P_\mathcal W))=\Codim(\Imm(I+P_\mathcal W-P_\mathcal V))<+\infty.
\]
This proves that $(\mathcal V^\perp,\mathcal W)\in\FP(\mathcal H)$, i.e., $\Comm(\mathcal H)\subseteq\Ecal(\mathcal H)$.
The proof of formula \eqref{commfred} is straigthforward.

To see that $\Comm(\mathcal H)$ actually does not coincide with $\Ecal(\mathcal H)$ consider the following example.
Let $H$ be a real infinite dimensional separable Hilbert space,
and set
\[
\mathcal H=H\times H, \quad \mathcal V=H\times\{0\} \quad \text{and} \quad \mathcal W=\big\{(x,x):x\in H\big\}.
\]
Obviously, $\mathcal V\cap \mathcal W=\{0\}$ and $\mathcal V+\mathcal W=\mathcal H$, so that $(\mathcal V,\mathcal W)\in\FP(\mathcal H)$
and $(\mathcal V,\mathcal W^\perp)\in\Ecal(\mathcal H)$.
An immediate calculation shows that $P_\mathcal VP_\mathcal W:\mathcal H\to\mathcal H$ is
given by $P_\mathcal VP_\mathcal W(a,b)=\frac12(a+b)$, which is clearly not a compact operator
on $\mathcal H$, so $(\mathcal V,\mathcal W^\perp)\not\in\Comm(\mathcal H)$.
\end{proof}

The following results will be useful in the sequel.

\begin{prop}\cite[Prop. 2.3.2]{Ab}.\label{abbo}
Given two self-adjoint Fredholm operators $L$ and $L'$ such that $L-L'$ is compact, the negative (resp. the positive) eigenspaces are commensurable.
\end{prop}

\begin{prop}\cite[Prop. 2.3.6]{Ab}.
\label{abbo236}
Let $B$ be a symmetric Fredholm form on $\mathcal H$ and $T$ the self-adjoint Fredholm operator associated with $B$.
Let $\mathcal V$ be a closed subspace of $\mathcal H$. Suppose  that $B$ is negative definite on $\mathcal V$ and positive semidefinite on $\mathcal V\ppb$. Then $(\mathcal V, V^+(T)\oplus \ker T)$ is a Fredholm pair of index zero.
\end{prop}


\subsection{Relative dimension of negative eigenspaces}

Let us recall that $B$ denotes a symmetric Fredholm form on the Hilbert space $\mathcal H$ and $T$ is the self-adjoint Fredholm operator
associated with $B$.

Let $\mathcal V$ be a closed subspace of $\mathcal H$ of finite codimension.
Call $\widetilde T=P_{\mathcal V}\circ T|_{\mathcal V}:\mathcal V\to \mathcal V$ the linear operator associated  with $B|_{\mathcal V\times \mathcal V}$, which is clearly a self-adjoint Fredholm operator (since $B|_{\mathcal V\times \mathcal V}$ is symmetric).

Recalling the spectral decomposition \eqref{decspet} of $\mathcal H$, induced by $B$, in this subsection we prove that $V^-(T)$ and $V^-(\wt )$ are commensurable and we give some results concerning the relative dimension $\dim(V^-(T),V^-(\wt))$ in different particular cases. The most general case, when $\mathcal V$ is any finite codimensional subspace of $\mathcal H$, will be tackled in Proposition \ref{final} below.

\begin{prop}
\label{spazicomm}
Given $T$ and $\wt$ as above, $V^-(T)$ and $V^-(\wt)$ are commensurable.
\epr

\begin{proof}
Define $\wh :\mathcal H\to \mathcal H$ as $\wh:=i\circ \wt\circ P_{\mathcal V}$ where $i:\mathcal V\hookrightarrow \mathcal H$ is the inclusion. It is immediate to see that $\wh$ is a Fredholm operator of index zero. In fact, the index of $P_{\mathcal V}$ coincides with the codimension of $\mathcal V$ in $\mathcal H$, while $\ind\, i=-\codim \mathcal V$. It is known that the composition of Fredholm operators is a Fredholm operator whose index is the sum of the indices of the components.

In the decomposition $\mathcal H=\mathcal V\oplus \mathcal V\pp$ we can represent $\wh$ in the block-matrix form as
\[
\wh=\left(
\begin{array}{cc}
\wt & 0\\
0 & 0\\
\end{array}
\right).
\]
As $\wt$ is self-adjoint, so is $\wh$.

Since $\mathcal V$ has finite codimension in $\mathcal H$, it follows that $T-\wh$ is compact. Indeed, consider the block-matrix representation of $T$ in the splitting $\mathcal H=\mathcal V\oplus \mathcal V\pp$:
\[
T=\left(
\begin{array}{cc}
\wt & T_{12}\\
T_{21} & T_{22}\\
\end{array}
\right),
\]
where $T_{12}=P_\mathcal V\circ T|_{\mathcal V\pp}$, $T_{21}=P_{\mathcal V\pp}\circ T|_\mathcal V$ and $T_{22}=P_{\mathcal V\pp}\circ T|_{\mathcal V\pp}$. These three operators have finite dimensional image. Therefore,
\[
T-\wh=\left(
\begin{array}{cc}
0 & T_{12}\\
T_{21} & T_{22}\\
\end{array}
\right)
\]
turns out to have finite dimensional image, and then it is compact.
We obtain, by Proposition \ref{abbo},  that $V^-(T)$ and $V^-(\wh)$ are commensurable.

Consider now the spectral decompositions of $\mathcal H$ induced by $\wh$ and of $\mathcal V$ induced by $\wt$:
\[
\mathcal H=V^-(\wh)\oplus V^+(\wh)\oplus \ker \wh, \quad
\mathcal V=V^-(\wt)\oplus V^+(\wt)\oplus \ker \wt.
\]
Since $\ker \wh = \ker \wt\oplus \mathcal V\pp$, we have
\[
\mathcal H=V^-(\wt)\oplus V^+(\wt)\oplus \ker \wh.
\]

The Fredholm form associated with $\wh$ is negative definite on $V^-(\wt)$ and positive on $V^+(\wt)$, as the definition of $\wh$ immediately shows. In addition both the spaces are invariant with respect to $\wh$. Therefore, by the uniqueness of the spectral decomposition, the above formula is actually the spectral decomposition of $\mathcal H$ by $\wh$, that is,
\[
V^-(\wh)=V^-(\wt) \quad  {\rm and} \quad V^+(\wh) =V^+(\wt).
\]
We have seen that $V^-(T)$ and $V^-(\wh)$ are commensurable. Of course, so are $V^-(T)$ and $V^-(\wt)$ and the proof is complete.
\end{proof}

Lemmas \ref{reldim} and \ref{forindgen} below give an answer to the question concerning the  relative dimension of $(V^-(T),V^-(\wt))$ in two particular cases. These results are interesting in themselves and propaedeutic to Proposition \ref{final}.

\ble{reldim}
Suppose $\mathcal H=\mathcal V+\vpb$. Let $T_2:=P_{\vpb}\circ \, T|_{\vpb}:\vpb\to \vpb$ be the linear operator associated  with $ B|_{\vpb \times \vpb}$. One has
\[
\dim(V^-(T),V^-(\wt))= \dim V^-(T_2).
\]
\ele

\begin{proof} It is immediate to see that $\mathcal V\cap \vpb$ is an isotropic space for $B$. Hence
\[
V^\pm(\wt)\cap V^\pm(T_2)=V^\pm(\wt)\cap \ker T_2 = \ker \wt\cap V^\pm(T_2)=\{0\}.
\]
Thus, given
\beq{vcappello}
\widehat V^-:=V^-(\wt)\oplus V^-(T_2) \quad {\rm and}\quad \widehat V^+:=V^+(\wt)\oplus V^+(T_2),
\eeq
and recalling that $\ker \wt\subseteq \ker T_2=\ker T$ (Lemma \ref{bortsommadir}), we have
\[
H=\widehat V^-\oplus \widehat V^+\oplus \ker T.
\]
Let us show that:
\begin{itemize}
\item [a)] $B$ is negative definite on $\widehat V^-$ and positive on $\widehat V^+$;\smallskip
\item [b)] $\widehat V^-$ and  $\widehat V^+$ are $B$-orthogonal.
\end{itemize}
a) Let $x\in\widehat V^-$ be given and write $x=x_1+x_2$ in the splitting $\widehat V^-= V^-(\wt) \oplus V^-(T_2)$. We have
\[
\ra Tx,x\la= \ra Tx_1+Tx_2,x_1+x_2\la=\ra Tx_1,x_1\la+\ra Tx_2,x_2 \la=\ra \wt x_1,x_1\la+\ra T_2x_2,x_2 \la
\]
(notice that $\ra Tx_1,x_2\la=0=\ra Tx_2,x_1\la$ since $x_1\in \mathcal V$ and $x_2\in \vpb$). The last two summands are, by definition of $\widehat V^-$, less or equal to zero, and not both zero if $x\neq 0$. Then $B$ is negative definite on $\widehat V^-$. The proof of the analogous result for $\widehat V^+$ is identical and omitted.

b) Let $x\in \widehat V^-$ and  $y\in\widehat V^+$ be given. By the decompositions
\eqref{vcappello},  write $x=x_1+x_2$ and $y=y_1+y_2$. Hence
\[
\ra Tx,y\la= \ra Tx_1+Tx_2,y_1+y_2\la=\ra Tx_1,y_1\la+\ra Tx_2,y_2 \la=\ra \wt x_1,y_1\la+ \ra T_2x_2, y_2  \la=0.
\]
The last equality is due to the fact that $V^-(\wt)$ and $V^+(\wt)$ are ($B|_{\mathcal V \times \mathcal V}$)-orthogonal, while $V^-(T_2)$ and $V^+(T_2)$ are ($B|_{\vpb \times \vpb}$)-orthogonal.

\medskip

We are now in the position to apply Proposition \ref{abbo236} to the pair $(\widehat V^-,V^+(T)\oplus \ker T)$ obtaining that it is a Fredholm pair of index zero.

Observe that $\widehat V^-$ and $V^-(T)$ are commensurable. Indeed $V^-(\wt)$ and $V^-(T)$ are commensurable by Proposition \ref{spazicomm}; in addition $\widehat V^-$ and $V^-(\wt)$ are of course commensurable since $V^-(T_2)$ has finite dimension. Now, recalling that $V^-(T)$ is the orthogonal complement of $V^+(T)\oplus \ker T$, by formula \eqref{commfred} it follows
\[
\dim(\widehat V^-,V^-(T))=0.
\]
In addition, it is immediate to see that
\[
\dim(\widehat V^-,V^-(\wt))=\dim V^-(T_2).
\]
By Lemma \ref{catena} we have
\[
\dim(V^-(T),V^-(\wt))= \dim V^-(T_2)
\]
and the proof is complete.
\end{proof}

\ble{forindgen}
Let $\mathcal Z$ be a finite dimensional subspace of $\mathcal H$, isotropic with respect to $B$, and call $L:\mathcal Z\ppb\to \mathcal Z\ppb$ the operator associated with $B|_{\mathcal Z\ppb\times \mathcal Z\ppb}$. Then $V^-(T)$ is commensurable with $V^-(L)$ and
\[
\dim(V^-(T),V^-(L))=\dim \mathcal Z -\dim (\mathcal Z\cap \ker T).
\]
\ele

\begin{proof}
Since $\mathcal Z$ is isotropic, we have $\mathcal Z\subseteq \mathcal Z\ppb$. Observe that $\mathcal Z\ppb$ is the orthogonal complement of $T(\mathcal Z)$ in $\mathcal H$. Therefore the codimension of $\mathcal Z\ppb$ in $\mathcal H$ is finite and
\beq{aster}
\codim \mathcal Z\ppb=\dim \mathcal Z-\dim (\ker T\cap \mathcal Z).
\eeq

The kernel of  $B|_{\mathcal Z\ppb\times\mathcal  Z\ppb}$ is
\[
\begin{aligned}
\ker B|_{\mathcal Z\ppb\times\mathcal  Z\ppb} & =\{x\in\mathcal  Z\ppb: \ra Tx,y\la=0, \; \forall y\in \mathcal Z\ppb\} \\
\, & = \{x\in \mathcal Z\ppb: \ra x,Ty\la=0, \; \forall y\in \mathcal Z\ppb\}.
\end{aligned}
\]
That is, $\ker B|_{\mathcal Z\ppb\times \mathcal Z\ppb}$ is a subspace of the orthogonal complement of $T(\mathcal Z\ppb)$ in $\mathcal H$. Hence, taking into account \eqref{aster}, one has
\[
\dim\ker B|_{\mathcal Z\ppb\times \mathcal Z\ppb}\leq \dim\mathcal  Z-\dim (\ker T\cap\mathcal  Z)+\dim \ker T.
\]
Since $\mathcal Z$ is isotropic, we have that $\mathcal Z\subseteq\ker B|_{\mathcal Z\ppb\times Z\ppb}$. Of course $\ker T\subseteq \mathcal Z\ppb$. Since
\[
\dim(\mathcal Z+\ker T)=\dim \mathcal Z-\dim (\ker T\cap \mathcal Z)+\dim \ker T,
\]
it follows
\[
\ker B|_{\mathcal Z\ppb\times \mathcal Z\ppb}=\mathcal Z+\ker T.
\]

The spectral decomposition of $\mathcal Z\ppb$ with respect to $B|_{\mathcal Z\ppb\times \mathcal Z\ppb}$ is
\[
\mathcal Z\ppb =V^-(L)\oplus V^+(L)\oplus (\mathcal Z+\ker T),
\]
and $\mathcal V:=V^-(L)\oplus V^+(L)$ is the orthogonal complement of $\mathcal Z+\ker T$ in $\mathcal Z\ppb$.
Observe that $B|_{\mathcal V\times \mathcal V}$ is  nondegenerate and then, by (2) in Lemma  \ref{bortsommadir},
\[
\mathcal H=\mathcal V\oplus \vpb.
\]
Since $\mathcal Z+\ker T$ is $B$-orthogonal to $\mathcal V$ it turns out to be contained in $\vpb$. An immediate computation says that
\[
\dim \vpb=2(\dim \mathcal Z-\dim(\mathcal Z\cap \ker T))+\dim \ker T.
\]

Call $\wt$ the operator associated with $B|_{\mathcal V\times \mathcal V}$ and $T_2$ that associated with $ B|_{\vpb\times \vpb}$. By Proposition \ref{reldim}, we have that $V^-(T)$ and $V^-(\wt)$ are commensurable and
\[
\dim(V^-(T),V^-(\wt))= \dim V^-(T_2).
\]

On the other hand $V^-(L)=V^-(\wt)$. Therefore, the proof is complete if we show that $\dim V^-(T_2)=\dim \mathcal Z-\dim (\mathcal Z\cap \ker T)$.

It is crucial now to notice that $\mathcal Z\subseteq \vpb$; this immediately follows from the inclusion $\mathcal Z+\ker T\subseteq \vpb$. By Proposition \ref{isomorse} we have, since $\mathcal Z$ is isotropic,
\[
\dim V^+(T_2)\geq \dim \mathcal Z -\dim (\mathcal Z\cap \ker T), \quad  \dim V^-(T_2)\geq \dim \mathcal Z- \dim (\mathcal Z\cap \ker T).
\]
Then
\[
\dim V^+(T_2)=\dim \mathcal Z-\dim (\mathcal Z\cap \ker T)=\dim V^-(T_2)
\]
and the proof is complete.
\end{proof}

We are now in the position to present the main result of this section, concerning the relative dimension of the negative eigenspaces of a self-adjoint Fredholm operator and its restriction to any closed finite codimensional subspace of $\mathcal H$.

\begin{prop}\label{final}
Let $B$ be a Fredholm symmetric bilinear form on $\mathcal H$ and
let $\mathcal V$ be a closed finite codimensional subspace of $\mathcal H$.
Denote by $T:\mathcal H\to \mathcal H$ and $\widetilde T=P_{\mathcal V}\circ T\vert_{\mathcal V}:{\mathcal V}\to {\mathcal V}$
the self-adjoint Fredholm operators associated with $B$ and $B|_{{\mathcal V}\times {\mathcal V}}$ respectively.
Then:
\[
\Dim\big(V^-(T),V^-(\widetilde T)\big)=\mathrm n_-\big(B|_{{\mathcal V}^{\perp_B}\times {\mathcal V}^{\perp_B}}\big)+\Dim({\mathcal V}\cap \vpb)
-\Dim\big({\mathcal V}\cap\ker T \big).\]
\end{prop}

\begin{proof}
Clearly $\mathcal Z:=\mathcal V\cap \vpb$ is an isotropic space. In addition it is finite dimensional since so is $\vpb$ (Remark \ref{osssulortog}, statement ii)). Let $R:\mathcal Z\ppb\to \mathcal Z\ppb$ be the linear operator associated with $B|_{\mathcal Z\ppb\times \mathcal Z\ppb}$. Then, by Lemma \ref{forindgen},
\[
\dim(V^-(T),V^-(R))=\dim \mathcal Z-\dim (\mathcal Z\cap \ker T).
\]

Now, as $\mathcal Z\ppb=\mathcal V+\vpb$ by statement (5) in Lemma  \ref{bortsommadir}, we can apply to $\mathcal Z\ppb$ Lemma \ref{reldim} and we obtain
\[
\dim(V^-(R),V^-(\wt))=\mathrm n_-(B|_{\vpb\times \vpb}).
\]
By Lemma \ref{catena} the claim follows.
\end{proof}
\end{section}


\begin{section}{On the spectral flow}
\label{sec:spflow}

\subsection{Generalities on the notion of spectral flow}\label{sectionsf}
Let us denote by $\Fredsa(\mathcal H)$ the set of self-adjoint Fredholm operators in $\mathcal H$.
Given a continuous path $T:[a,b]\to\Fredsa(\mathcal H)$, we will denote by $\spfl(T,[a,b])$
the \emph{spectral flow} of $T$ on the interval $[a,b]$, which is an integer
number that gives, roughly speaking, the net number of eigenvalues of $T$ that pass through the value $0$.

There
exist several equivalent definitions of the spectral flow in the literature, although the reader
should  note that there exist different conventions on the contribution of the endpoints in the
case when $T_a$ and/or $T_b$ are not invertible.

A possible definition of spectral flow using
functional calculus is given in \cite{Phillips} as follows. Let $t_0=a<t_1<\ldots<t_N=b$ be a partition
of $[a,b]$, and $a_1,\ldots,a_N$ be positive numbers with the property that, denoting by $\chi_I$ the
characteristic function of the interval $I$, for $i=1,\ldots,N$ the following hold:

\begin{itemize}
\item[(a)] the map $[t_{i-1},t_i]\ni t\mapsto\chi_{[-a_i,a_i]}(T_t)$ is continuous,
\item[(b)] $\chi_{[-a_i,a_i]}(T_t)$ is a projection onto a finite dimensional subspace of $\mathcal H$.
\end{itemize}
Then, $\spfl\big(T,[a,b]\big)$ is defined by the sum
\[
\spfl\big(T,[a,b]\big)=\sum_{i=1}^N=\big[\mathrm{rk}\big(\chi_{[0,a_i]}(T_{t_i})-\mathrm{rk}\big(\chi_{[0,a_i]}(T_{t_{i-1}})\big],
\]
where $\mathrm{rk}(P)$ denotes the rank of a projection $P$. With this definition, in the particular case when $T$ is a path of essentially positive operators, that is, the negative spectrum of each operator $T_t$ has only isolated eigenvalues of finite multiplicity, then the spectral flow of $T$ is given by
\[
\spfl(T,[a,b])=\mathrm n_-(T_b)+\Dim\big(\Ker T_b\big)-
\mathrm n_-(T_a)-\Dim\big(\Ker T_a\big).
\]

The spectral flow is additive by concatenation of paths,
and invariant by fixed-endpoints homotopies, and it therefore defines a $\Z$-valued homomorphism on
the fundamental group\-oid of $\Fredsa(\mathcal H)$. In fact, one shows easily that the spectral flow
is invariant by the larger class of homotopies that leave constant the dimension of the kernel at the endpoints.
Moreover, the spectral flow is invariant by cogredience, i.e., given Hilbert spaces $\mathcal H_1$,
$\mathcal H_2$, a continuous path $T:[a,b]\to\Fredsa(\mathcal H_2)$ and
a continuous path of isomorphisms $S:[a,b]\to\mathrm{Iso}(\mathcal H_1,\mathcal H_2)$, then the spectral flow
of the path $[a,b]\ni t\mapsto S_t^*T_tS_t\in\Fredsa(\mathcal H_1)$ equals the spectral flow of $T$.

\smallskip

We are interested in computing the spectral flow of paths of  self-adjoint Fredholm operators
that are compact perturbations of a fixed symmetry of the Hilbert space $\mathcal H$. By a \emph{symmetry}
of $\mathcal H$ we mean a bounded operator $\mathfrak I$ on $\mathcal H$ of the form $\mathfrak I=P_{\mathcal W}-P_{{\mathcal W}^\perp}=2P_{\mathcal W}-I$, where ${\mathcal W}$ is a given closed subspace of $\mathcal H$. Equivalently, $\mathfrak I$ is a symmetry if it is self-adjoint and it satisfies $\mathfrak I^2=I$, the identity map of $\mathcal H$.


A symmetry $\mathfrak I$ can be represented, with respect to the decomposition $\mathcal H=\mathcal W \oplus {\mathcal W}^\perp $, by the matrix
\[
\begin{pmatrix}
I_{\mathcal W} & 0 \\
0 & -I_{{\mathcal W}^\perp}
\end{pmatrix}\]
where $I_{\mathcal W}$ and $I_{{\mathcal W}^\perp}$ are the identity maps of $\mathcal W$ and ${\mathcal W}^\perp$, respectively.

A compact perturbation of $\mathfrak I$ is essentially positive, essentially negative or strongly indefinite according to whether $\mathcal W^\perp$ is finite dimensional, $\mathcal W$ is finite dimensional, or both $\mathcal W$ and $\mathcal W^\perp$ are infinite dimensional, respectively. Of course the last case could happen only if $\mathcal H$ is infinite dimensional.

Given a continuous curve $T:[a,b]\to\Fredsa(\mathcal H)$ of the form $T_t=\mathfrak I+K_t$, where $\mathfrak I$ is a symmetry of $\mathcal H$
and $K_t$ is a self-adjoint
compact operator on $\mathcal H$, then the spectral flow of $T$ can be computed in terms of the notion of  relative dimension, recalled in the above section, as follows: by Proposition \ref{abbo}
the spaces $V^-(T_a)$ and $V^-(T_b)$ are commensurable, and
\beq{sfdimrel}
\spfl\big(T,[a,b]\big)=\Dim\big(V^-(T_a),V^-(T_b)\big).
\eeq
Here comes an immediate observation, that will be useful ahead.

\begin{prop}\label{thm:dipendedagliestremi}
For a continuous path $T:[a,b]\to\Fredsa(\mathcal H)$ of the form $T_t=\mathfrak I+K_t$, where $\mathfrak I$ is a symmetry of $\mathcal H$
and $K_t$ is a self-adjoint compact operator on $\mathcal H$, the spectral flow $\spfl(T,[a,b])$ depends only on the
endpoints $T_a$ and $T_b$.\qed
\end{prop}

\subsection{Restriction to a fixed subspace}

An important property, stated in the following lemma, says that if $\mathcal V$ is a  closed subspace of $\mathcal H$ of finite codimension, then $P_{\mathcal V}\circ T_t|_\mathcal V:\mathcal V\to \mathcal V$ is a path of self-adjoint compact perturbations of a fixed symmetry of $\mathcal V$.

\ble{symmetryrestrictions}
Let $T:[a,b]\to\Fredsa(\mathcal H)$ be a continuous curve  of the form $T_t=\mathfrak I+K_t$, where $\mathfrak I$ is a symmetry of $\mathcal H$ and $K_t$ is a self-adjoint compact operator on $\mathcal H$, and consider a closed subspace $\mathcal V$ of $\mathcal H$ of finite codimension.
Call $\widetilde T:[a,b]\to\Fredsa(\mathcal V)$ the continuous curve of self-adjoint operators defined as $\widetilde T_t=P_{\mathcal V}\circ T_t|_\mathcal V$.
Then, there exist a symmetry $\mathfrak I_{\mathcal V}$ of $\mathcal V$ and a continuous path of self-adjoint
compact operators $C_t$ on $\mathcal V$ such that
\[
\wt_t=\mathfrak I_{\mathcal V}+C_t, \quad t\in[a,b].
\]
\ele

\begin{proof}
The operator $\widetilde{\mathfrak I}_{\mathcal V}=P_{\mathcal V}\circ\mathfrak I\vert_{\mathcal V}:\mathcal V\to \mathcal V$ is self-adjoint,
and its square $(\widetilde{\mathfrak I}_{\mathcal V})^2$ is easily computed as the sum of the identity of $\mathcal V$ and a finite rank operator.
Namely, the space $\mathcal W=\mathfrak I^{-1}(\mathcal V)\cap\mathcal V=\mathfrak I(\mathcal V)\cap\mathcal V$ has finite codimension in $\mathcal V$, it is invariant
by $\mathfrak I$, and $(\mathfrak I\vert_{\mathcal W})^2=I_{\mathcal W}$.
The symmetry $\mathfrak I_{\mathcal V}$ is obtained applying next Lemma to the operator $S=\widetilde{\mathfrak I}_{\mathcal V}$ on the
Hilbert space $\mathcal V$.
\end{proof}

\begin{lem}
Let $S$ be a self-adjoint operator on a Hilbert space $\mathcal G$ such that $S^2-I$ has finite rank. Then, $S$ is a finite-rank
perturbation of a symmetry $\mathfrak L$ of $\mathcal G$.
\end{lem}
\begin{proof}
$S^2-I$ is self-adjoint and it has closed image (finite dimensional), thus $\mathcal G$ is given by the orthogonal
sum of closed subspaces, that is, $\mathcal G=\Ker(S^2-I)+\mathrm{Im}(S^2-I)$. The symmetry $\mathfrak L$ is given by
\[
\mathfrak L=\begin{cases}S &\text{on $\Ker(S^2-I)$}\\
{I} & \text{on $\mathrm{Im}(S^2-I)$}.\end{cases}\qedhere\]
\end{proof}

We are now in the position to present the followin result, which concerns the difference between the spectral flow of a path of symmetric Fredholm forms
on $\mathcal H$ and the spectral flow of its restriction to a finite codimensional closed subspace of $\mathcal H$.

In the theorem $\Bsym(\mathcal H)$ will denote the set of symmetric Fredholm forms
on $\mathcal H$, while $\Fredsa(\mathcal H)$, as said before, will stand for the set of self-adjoint Fredholm operators in $\mathcal H$.

\bth{riduzionespettro}
Consider a continuous path $B:[a,b]\to\Bsym(\mathcal H)$ of symmetric Fredholm forms
on $\mathcal H$.
Let ${\mathcal V}$ be a finite codimensional closed subspace of $\mathcal H$ and denote by $T:[a,b]\to\Fredsa(\mathcal H)$ and $\widetilde T:[a,b]\to\Fredsa(\mathcal V)$ the continuous paths of self-adjoint Fredholm operators
associated with $B$ and to the restriction $B|_{\mathcal V\times\mathcal V}$, respectively.
Assume that $T_t=\mathfrak J+K_t$ for all $t\in[a,b]$, where $\mathfrak J$ is a symmetry of $\mathcal H$ and $K_t$ is compact for all $t$. Then,
\begin{equation}\label{primariduzione}
\begin{aligned}
\spfl(T,[a,b])&-\spfl(\widetilde T,[a,b])=\Dim\big(V^-(T_a),V^-(\widetilde T_a)\big)-
\Dim\big(V^-(T_b),V^-(\widetilde T_b)\big)\\
&
=\mathrm n_-\big(B_a\vert_{\mathcal V^{\perp_{B_a}}\times\mathcal V^{\perp_{B_a}}}\big)+\Dim\big(\mathcal V\cap\mathcal V^{\perp_{B_a}}\big)-\Dim\big(\mathcal V\cap\Ker B_a\big)\\
&\quad-\mathrm n_-\big(B_b\vert_{\mathcal V^{\perp_{B_b}}\times \mathcal V^{\perp_{B_b}}}\big)-\Dim\big(\mathcal V\cap\mathcal V^{\perp_{B_b}}\big)+\Dim\big(\mathcal V\cap\Ker B_b\big).
\end{aligned}
\end{equation}
\eth

\begin{proof}
By Lemma \ref{symmetryrestrictions} $\widetilde T$ is a path of compact perturbations of a symmetry of $\mathcal V$. Therefore, by formula \eqref{sfdimrel} we immediately obtain
\[
\spfl(T,[a,b])-\spfl(\widetilde T,[a,b])=\Dim\big(V^-(T_a),V^-(T_b)\big)-
\Dim\big(V^-(\widetilde T_a),V^-(\widetilde T_b)\big).
\]
Recalling that the commensurability of subspaces is an equivalence relation and applying Lemma
\ref{catena}, it follows that
\[
\begin{array}{l}
\Dim\big(V^-(T_a),V^-(T_b)\big)-
\Dim\big(V^-(\widetilde T_a),V^-(\widetilde T_b)\big)=\smallskip \\
\Dim\big(V^-(T_a),V^-(\widetilde T_a)\big)-
\Dim\big(V^-(T_b),V^-(\widetilde T_b)\big).
\end{array}
\]
The conclusion of the proof is an immediate consequence of Proposition \ref{final}.
\end{proof}

Note that $\mathcal V\cap\mathcal V^{\perp_B}=\Ker\big(B\vert_{\mathcal V\times\mathcal V}\big)$.

\subsection{Continuous and smooth families of closed subspaces}
\label{sub:contclosedsubspaces}
In Subsection \ref{subultima} below we will extend formula \ref{primariduzione} to the case when the subspace $\mathcal V$ in Theorem \ref{riduzionespettro} is not constant but depends on $t$. To this end we devote this subsection to a summary of the concept of smooth family (or smooth path) of closed
subspaces of $\mathcal H$, recalling also some crucial properties, important for our construction.
The goal is to determine the existence of a special class
of \emph{trivializations} for smooth, or continuous, curves of closed subspaces.
Most of the material discussed in this subsection is known to specialists, nevertheless it will be useful
to give a formal proof of the essential results, for the reader's convenience.

In the following definition, being $\mathrm{L}(\mathcal H)$ the space of bounded linear operators of $\mathcal H$ into itself, $\mathrm{GL}(\mathcal H)$ is the open subset of $\mathrm{L}(\mathcal H)$ of the automorphisms. The space of bounded linear operators between two Hilbert spaces $\mathcal H_1$ and $\mathcal H_2$ is denoted by $\mathrm L(\mathcal H_1,\mathcal H_2)$.

\begin{defin}\label{thm:defsmoothnessclosedsubspaces}
Let $I\subseteq\R$ be an interval and $\mathcal D=\{\mathcal V_t\}_{t\in I}$ be a family of closed subspaces of
$\mathcal H$. We say that $\mathcal D$ is a \emph{$C^k$ family of closed subspaces of $\mathcal H$}, $k=0,\ldots,\infty,\omega$\footnote{The symbol $C^\omega$ means analytic.}
if for all $t_0\in I$ there exist $\varepsilon>0$, a $C^k$ map $\Psi: I\cap\left]t_0-\varepsilon,t_0+\varepsilon\right[\to\mathrm{GL}(\mathcal H)$ and a closed subspace $\mathcal V_\star\subseteq\mathcal H$ such that
$\Psi_t(\mathcal V_t)=\mathcal V_\star$ for all $t\in I\cap\left]t_0-\varepsilon,t_0+\varepsilon\right[$.
\end{defin}

The pair $(\mathcal V_\star,\Psi)$ as above will be called a \emph{$C^k$-local trivialization} of the family $\mathcal D$ around $t_0$.
The following criterion of smoothness holds.

\begin{prop}\label{thm:produce}
Let $I\subseteq\R$ be an interval, $\mathcal H_1$, $\mathcal H_2$ be Hilbert
spaces and $F:I\mapsto\mathrm L(\mathcal H_1,\mathcal H_2)$ be a $C^k$ map, $k=0,1,\ldots,\infty,\omega$,
such that each $F_t$ is surjective. Then, the family of $\mathcal V_t=\ker F_t$
is a $C^k$-family  of closed subspaces of $\mathcal H_1$.
\end{prop}

\begin{proof}
See for instance \cite[Lemma~2.9]{asian}.
\end{proof}

Let $\mathcal D=\{\mathcal V_t\}_{t\in I}$ be a family of closed subspaces of $\mathcal H$.
Proposition \ref{thm:smoothnessequivalent} below relates the smoothness of $\mathcal D$  with the
smoothness of the path  $t\mapsto P_{\mathcal V_t}$ of the orthogonal projections onto $\mathcal V_t$, for $t\in I$. Any $P_{\mathcal V_t}$ is considered having $\mathcal H$ as target space. We need first two preliminary lemmas.

\begin{lem}\label{thm:lemprojprox}
Let $P,Q$ be two projections such that $\Vert P-Q\Vert<1$.
Then, the restriction $P^{\Im Q}_{\Im P}:\Im Q\to\Im P$ is an isomorphism.
\end{lem}

\begin{proof}
Assume $x\in\Im Q\setminus\{0\}$ and $Px=0$; then $\Vert Px-Qx\Vert=\Vert Qx\Vert=\Vert x\Vert$, which implies
$\Vert P-Q\Vert\ge1$. Thus, $P^{\Im Q}_{\Im P}$ is injective. We now need to show that $\Im P^{\Im Q}_{\Im P}$
is equal to $\Im P$; to this aim, it suffices to show that $\Im(PQ)=\Im P$.
This follows easily from the equality
\[
PQ=P(Q+ I-P),
\]
observing that, since $\Vert P-Q\Vert<1$, then $ I+Q-P$ is an isomorphism of $\mathcal H$.
\end{proof}

\begin{lem}\label{thm:lemmaproiezionegrafo}
Let $\mathcal H_0$ and $\mathcal H_1$ be Hilbert spaces, and let $L:\mathcal H_0\to\mathcal H_1$ be a bounded linear
operator. Set $\mathcal H=\mathcal H_0\oplus\mathcal H_1$; then, the orthogonal projection
$P_{\Gr(L)}$ onto the graph of $L$ is given by

\begin{equation}\label{eq:exprorthprojgraph}
\begin{aligned}
& P_{\Gr(L)}(x,y)= & \\ & \big(x+L^*(I+LL^*)^{-1}(y-Lx), L(x+L^*( I+LL^*)^{-1}(y-Lx))\big)= &\\
& \big(x+L^*( I+LL^*)^{-1}(y-Lx), y-( I+LL^*)^{-1}(y-Lx)\big). &
\end{aligned}
\eeq
\end{lem}

\begin{proof}
It follows by a straightforward calculation, keeping in mind that the orthogonal complement of $\Gr(L)$ in $\mathcal H$
is $\big\{(-L^*b,b):b\in \mathcal H_1\big\}$.
\end{proof}

Formula \eqref{eq:exprorthprojgraph} shows that the orthogonal projection onto the graph of $L$ is written as a smooth
function of $L$. We are now ready for the main result of the subsection.

\begin{prop}\label{thm:smoothnessequivalent}
Let $J\subseteq\R$ be an interval, and let $\mathcal D=\{\mathcal V_t\}_{t\in J}$ be a family of closed
subspaces of $\mathcal H$. Then, for all $k=0,1,\ldots,\infty,\omega$, $\mathcal D$ is a $C^k$-family of subspaces of $\mathcal H$ if and only if the map
$t\mapsto P_{\mathcal V_t}$, from $J$ into $\in\mathrm{L}(\mathcal H)$,
is of class $C^k$.
\end{prop}

\begin{proof}
Assume that $t\mapsto P_{\mathcal V_t}$ is of class $C^k$; set $Q_t=  I-P_{\mathcal V_t}$, so that $\mathcal V_t=\Ker Q_t$ for
all $t$. Fix $t_0\in J$, for $t\in J$ near $t_0$, by continuity we can assume $\Vert Q_t-Q_{t_0}\Vert<1$.
We claim that, for $t\in J$ near $t_0$, the map $F_t=Q_{t_0}Q_t:\mathcal H\to\Im Q_{t_0}$ is surjective;
namely, $\Im F_t=\mathrm{Im}\big(Q_{t_0}\vert_{\Im Q_t}\big)$, and the claim follows
from Lemma~\ref{thm:lemprojprox}. Moreover, $\Ker F_t=\Ker Q_t$ because, by Lemma~\ref{thm:lemprojprox},
$Q_{t_0}\vert_{\Im Q_t}$ is injective. Since $t\mapsto F_t$ is of class $C^k$,
$\mathcal D$ is a $C^k$-family of closed subspaces of $\mathcal H$ by Proposition~\ref{thm:produce}.

For the converse, we will show that the projections $P_{\mathcal V_t}$ can be written as smooth functions
of a local trivialization.
Assume $\mathcal D$ of class $C^k$; choose $t_0\in J$, and let
$(\mathcal V_\star,\Psi)$ be a local trivialization  of $\mathcal D$ around $t_0$; set $\phi_t=\Psi_t^{-1}$.
Up to replacing $\Psi_t$ with $\Psi_{t_0}^{-1}\Psi_t$, we can assume $\mathcal V_\star=\mathcal V_{t_0}$ and
$\mathcal V_t=\phi_t(\mathcal V_{t_0})$ for all $t$ near $t_0$. Write $\mathcal H=\mathcal V_{t_0}\oplus\mathcal V_{t_0}^\perp$
and write $\phi_t$ in blocks relatively to this decomposition of $\mathcal H$ as:
\[
\phi_t=\begin{pmatrix}\phi^{11}_t&\phi^{12}_t\\\phi^{21}_t&\phi^{22}_t \end{pmatrix};
\]
observe that the smoothness of $\Psi_t$ is equivalent to the smoothness of the blocks $\phi^{ij}_t$.
Since $\phi_{t_0}\vert_{\mathcal V_{t_0}}$ is the identity on $\mathcal V_{t_0}$,
$\phi^{11}_{t_0}$ is the identity, and by continuity, $\phi^{11}_t$ is invertible for $t$ near
$t_0$. An immediate computation shows that, setting $L_t:\mathcal V_{t_0}\to\mathcal V_{t_0}^\perp$,
\[L_t=\phi^{21}_t\circ(\phi^{11}_t)^{-1},\]
then $\mathcal V_t=\Gr(L_t)$. Using Lemma~\ref{thm:lemmaproiezionegrafo}, the projection
$P_{\mathcal V_t}$ onto $\mathcal V_t$ can be written as a smooth function of $\phi_t$, which proves
that $t\mapsto P_{\mathcal V_t}$ is of class $C^k$.
\end{proof}





\bre{trivializzazionediagonale}
The above proposition  tells us that, given a $C^k$ family $\mathcal D=\{\mathcal V_t\}_{t\in[a,b]}$ of closed subspaces of $\mathcal H$, there exists, for any $t_0\in [a,b]$, a local trivialization $(\mathcal V_\star,\Psi)$ of $\mathcal D$ around $ t_0$ such that
$\Psi_t(\mathcal V_t^\perp)=\mathcal V_\star^\perp$
for all $t$ in a neighborhood $I$ of $ t_0$.
\ere

\bde{splitriv}
A local trivialization $(\mathcal V_\star,\Psi)$ of $\mathcal D$ around $ t_0$ is called a \emph{local splitting trivialization} if $\Psi_t (\mathcal V_t^\perp)=\mathcal V_\star^\perp$.
\ede

Actually,
 as an immediate consequence of Corollary \ref{thm:esistenzatrivializzazioniortogonali}, we obtain the following \emph{global} result, that is, the existence of a global splitting trivialization of isometries.
\begin{prop}
Given a $C^k$ family $\mathcal D=\{\mathcal V_t\}_{t\in[a,b]}$ of closed subspaces of $\mathcal H$, there exists a global trivialization $(\mathcal V_\star,\Psi)$ of $\mathcal D$ such that
$\Psi_t\in \mathrm O(\mathcal H)$ for all $t \in [a,b]$.
\end{prop}

\subsection{Spectral flow and restrictions to a continuous family of subspaces}\label{subultima}
The additivity by concatenation of paths and invariance by cogredience properties, recalled in Subsection \ref{sectionsf},
allow us to extend the definition of spectral
flow to the case of paths of Fredholm operators with varying domains.

Assume that $[a,b]\ni t\mapsto T_t$ is a continuous map of bounded operators on $\mathcal H$ and
$\mathcal D=\{\mathcal V_t\}_{t\in[a,b]}$ is a continuous family of closed subspaces such that,
taking the orthogonal projection $P_{\mathcal V_{t}}$ as a map with target space ${\mathcal V_{t}}$ for every $t\in[a,b]$, the operator
$P_{\mathcal V_{t}}\circ T_t\vert_{\mathcal V_t}:\mathcal V_t\to\mathcal V_t$ is Fredholm and self-adjoint.
Let  $(\mathcal V_\star,\Psi)$
be a  trivialization of $\mathcal D$, and denote by $P_\star$ the orthogonal projection onto $\mathcal V_\star$. Then,  we have a continuous family $[a,b]\ni t\mapsto \widetilde T_t\in\Fredsa(\mathcal V_\star)$
of self-adjoint Fredholm operators on $\mathcal V_\star$, obtained by setting\footnote{%
$\big(\Psi_t\big\vert_{\mathcal V_t}\big)^*=P_{\mathcal V_t}\Psi_t^*\big\vert_{\mathcal V_\star}$.}
\begin{equation}
\label{eq:defwidetildeTt}
\widetilde T_t=P_\star\circ\big(\Psi_t\big\vert_{\mathcal V_t}\big) \circ P_{\mathcal V_{t}}\circ T_t\circ\big(\Psi_t\big\vert_{\mathcal V_t}\big)^*:\mathcal V_\star\longrightarrow
\mathcal V_\star.
\end{equation}

We define the spectral flow $\spfl(T,\mathcal D;[a,b])$ of the path $T=(T_t)_{t\in[a,b]}$ restricted to the varying domains $\mathcal D=(\mathcal V_t)_{t\in[a,b]}$ by
\begin{equation}\label{eq:defspflvaryngdom}
\spfl\big(T,\mathcal D;[a,b]\big)=\spfl\big(\widetilde T,[a,b]\big).
\end{equation}

In order to prove that this is a valid definition, one needs the following lemma.

\begin{lem}
The right hand side of equality \eqref{eq:defspflvaryngdom} does not depend on the choice of a
trivialization of the family $\mathcal D$.
\end{lem}

\begin{proof}
Assume that $(\widehat{\mathcal V}_\star,\widehat\Psi)$
is another trivialization of $\mathcal D$. Denoting by $\widehat P_\star$
the orthogonal projection onto $\widehat{\mathcal V}_\star$, set
\[
\widehat T_t=\widehat P_\star\circ\big(\widehat\Psi_t\big\vert_{\mathcal V_t}\big)
\circ P_{\mathcal V_{t}}\circ T_t\circ
\big(\widehat\Psi_t\big\vert_{{\mathcal V}_t}\big)^*:\widehat{\mathcal V}_\star\longrightarrow
\widehat{\mathcal V}_\star
\]
and denote by $\Phi_t:\mathcal V_\star\to\widehat{\mathcal V}_\star$ the isomorphism
$\big(P_{\mathcal V_t}\circ(\widehat\Psi_t^*\big\vert_{\widehat{\mathcal V}_\star})\big)^{-1}\circ P_{\mathcal V_t}\circ(\Psi_t^*\big\vert_{\mathcal V_\star})$.
If $\widetilde T_t$ is as in formulas \eqref{eq:defwidetildeTt}, then
\[
\widetilde T_t=\Phi_t^*\circ \widehat T_t\circ\Phi_t
\]
for all $t$, hence $\spfl(\widetilde T,[a,b])=\spfl(\widehat T,[a,b])$, by the cogredient invariance of the spectral flow.
\end{proof}

Our aim is to show how the result of Theorem~\ref{riduzionespettro} can be employed in the
computation of the spectral flow in the case of varying domains.
Towards this goal, we observe preliminarily that if $(\widehat{\mathcal V}_\star,\widehat\Psi)$ is an orthogonal trivialization of $\mathcal D$,
then $\Psi_t^*(\mathcal V_\star)=\mathcal V_t$ for\footnote{%
This holds more generally for splitting trivializations.} all $t$; this simplifies formula \eqref{eq:defwidetildeTt},
in that $(\Psi_t\vert_{\mathcal V_t})^*=\Psi_t^*\vert_{\mathcal V_\star}=\Psi_t^{-1}\vert_{\mathcal V_\star}$.
Moreover, it is easy to show that, given a continuous path $[a,b]\ni t\mapsto U_t$ with values in $\mathrm O(\mathcal H)$, the set of the orthogonal automorphisms of $\mathcal H$,
then the spectral flow of the path $[a,b]\ni t\mapsto T_t$ restricted to a continuous family of subspaces of $\mathcal H$,  $\mathcal D=\{\mathcal V_t\}_{t\in[a,b]}$ is equal
to the spectral flow of the path $[a,b]\ni t\mapsto U_tT_tU_t^*$ restricted to the family $\{U_t(\mathcal V_t)\}_{t\in[a,b]}$.

We are now ready for the following result.

\begin{prop}\label{thm:propsubs}
Let $T:[a,b]\to\Fredsa(\mathcal H)$ be a continuous path of the form $T_t=\mathfrak I+K_t$, where $\mathfrak I$
is a symmetry of $\mathcal H$ and $K_t$ is, for any $t\in [a,b]$, a self-adjoint compact operator on $\mathcal H$.
Consider a continuous family $\mathcal D=\{\mathcal V_t\}_{t\in[a,b]}$  of (finite codimensional) closed subspaces of $\mathcal H$.
Then, there exists an orthogonal trivialization $(\mathcal V_\star,\Psi)$ of $\mathcal D$ (with $\mathcal V_\star$ finite codimensional)
and a symmetry $\widetilde{\mathfrak J}:\mathcal H\to\mathcal H$
such that $\Psi_tT_t\Psi_t^*-\widetilde{\mathfrak J}$ is compact for all $t\in[a,b]$.
\end{prop}

\begin{proof}
Choose any orthogonal trivialization $({\mathcal V}_\star,\Phi)$ of $\mathcal D$, so that by what has been just observed,
the spectral flow of $T$ restricted to  $\mathcal D$ equals the spectral flow of $t\mapsto\Phi_tT_t\Phi_t^*=\Phi_t\mathfrak I\Phi_t^*+\Phi_tK_t\Phi_t^*$
restricted to the fixed subspace $\Phi_t(\mathcal V_t)=\mathcal V_\star$.

Since $\Phi_t$ is orthogonal, then, for all $t$, $\widehat{\mathfrak I}_t=\Phi_t\mathfrak I\Phi_t^*$ is a symmetry of $\mathcal H$;
the operator $\Phi_tK_t\Phi_t^*$ is clearly compact.
By Lemma~\ref{symmetryrestrictions}, if $P_\star$ is the
orthogonal projection onto $\mathcal V_\star$, the operator $P_\star\Phi_tT_t\Phi_t^*\vert_{\mathcal V_\star}\in\mathcal F_{\mathrm{sa}}(\mathcal V_\star)$
is of the form $\mathfrak I_t^\star+C_t$, where $t\mapsto\mathfrak I^\star_t$ is a continuous path of symmetries of the Hilbert space  $\mathcal V_\star$.
Now, by Corollary~\ref{thm:esistenzatrivializzazioniortogonalisym}, there exists a continuous path $t\mapsto U_t\in\mathrm O(\mathcal V_\star)$
and a fixed symmetry $\mathfrak J_\star$ of $\mathcal V_\star$ with the property that $U_t\mathfrak I_t^\star U_t^*=\mathfrak I_\star$ for all $t$.
Extend $\mathfrak J_\star$ to a symmetry $\widetilde{\mathfrak J}$ of $\mathcal H$ by setting $\widetilde{\mathfrak J}\vert_{\mathcal V_\star^\perp}$
equal to the identity, and each $U_t$ to an orthogonal operator $W_t\in\mathrm O(\mathcal H)$ by setting $W_t\vert_{\mathcal V_\star^\perp}$ equal to
the identity. Observe that $W_t$ commutes with $P_\star$ for all $t$, since $\mathcal V_\star$ is $W_t$-invariant.
Then, the required trivialization of $\mathcal D$ is obtained by setting $\Psi_t=W_t\Phi_t$ for all $t$.
\end{proof}

Using an orthogonal trivialization as in Proposition~\ref{thm:propsubs}, Theorem~\ref{riduzionespettro} can now be employed in
the computation of the spectral flow of restrictions to a varying family of finite codimensional subspaces.

\bth{riduspettrofamiglia}
Consider a continuous path $B:[a,b]\to\Bsym(\mathcal H)$ of symmetric Fredholm forms
on $\mathcal H$ and denote by $T:[a,b]\to\Fredsa(\mathcal H)$  the continuous paths of self-adjoint Fredholm operators
associated with $B$. Consider a continuous family $\mathcal D=\{\mathcal V_t\}_{t\in[a,b]}$  of (finite codimensional) closed subspaces of $\mathcal H$ and let $(\mathcal V_\star,\Psi)$ be an orthogonal trivialization of $\mathcal D$ and  $\widetilde{\mathfrak J}:\mathcal H\to\mathcal H$ be a symmetry such that $\Psi_tT_t\Psi_t^*-\widetilde{\mathfrak J}$ is compact for all $t\in[a,b]$.
Denote by $\widetilde T:[a,b]\to\Fredsa(\mathcal V_\star)$ the path $\widetilde T_t=P_\star\Psi_tT_t\Psi_t^*\vert_{\mathcal V_\star}$, where $P_{\star}$ is the projection onto $\mathcal V_\star$.

Then, we have
\begin{equation}\label{riduzionespazivariabili}
\spfl(T,[a,b])-\spfl\big(T,\mathcal D;[a,b]\big)
\!=\!\Dim\big(V^-(T_a),V^-(\widetilde T_a)\big)-
\Dim\big(V^-(T_b),V^-(\widetilde T_b)\big).
\end{equation}
\eth

\begin{proof}
Denote  $\widehat T_t=\Psi_tT_t\Psi_t^*:\mathcal H\to\mathcal H$, for any $t\in [a,b]$. Since $T$ and $\widehat T$ are cogredient, their spectral flows coincide, and, since $\widehat T$ is a path of compact perturbations of a symmetry, we have
\[
\spfl(T,[a,b])=\Dim\big(V^-(T_a),V^-(T_b)\big)= \spfl(\widehat T,[a,b])=\Dim\big(V^-(\widehat T_a),V^-(\widehat T_b)\big).
\]

Applying Theorem \ref{riduzionespettro}, we have
\[
\spfl(\widehat T,[a,b])-\spfl(\widetilde T,[a,b])=\Dim\big(V^-(\widehat T_a),V^-(\widetilde T_a)\big)-
\Dim\big(V^-(\widehat T_b),V^-(\widetilde T_b)\big),
\]
and, finally, by Lemma \ref{catena} the claim follows.
\end{proof}

\end{section}

\begin{section}{Spectral flow along periodic semi-Riemannian geodesics}
\label{sec:applicationsRGeom}
In this section we will discuss an application to semi-Riemannian geometry
of our spectral flow formula. We will define the \emph{spectral flow} of the
index form along a periodic geodesic in a semi-Riemannian manifold, and
we will compute its value in terms of the Maslov index of the geodesic.
In the Riemannian (i.e., positive definite) case, the spectral flow is
equal to the Morse index of the geodesic action functional at the
closed geodesic, and the Maslov index is given by the number of conjugate points
along a geodesic. In the general semi-Riemannian case, it is well known that
the Morse index of the geodesic action functional is infinite.

\subsection{Periodic geodesics}
We will consider throughout an $n$-dimensional semi-Riemannian manifold
$(M,g)$, denoting by $\nabla$ the covariant derivative of its Levi--Civita
connection, and by $R$ its curvature tensor, chosen with the sign convention
$R(X,Y)=[\nabla_X,\nabla_Y]-\nabla_{[X,Y]}$.

Let $\gamma:[0,1]\to M$ be a periodic geodesic in $M$, i.e., $\gamma(0)=\gamma(1)$ and
$\dot\gamma(0)=\dot\gamma(1)$. We will assume that $\gamma$ is \emph{orientation
preserving}, which means that the parallel transport along $\gamma$ is orientation preserving.
If $M$ is orientable, then every closed geodesic is orientation preserving. Moreover,
given any closed geodesic $\gamma$, its two-fold iteration $\gamma^{(2)}$, defined
by $\gamma^{(2)}(t)=\gamma(2t)$, is always orientation preserving.

We will denote by $\Ddt$ the covariant differentiation of vector fields along $\gamma$;
recall that the \emph{index form} $I_\gamma$ is the bounded symmetric
bilinear form defined on the Hilbert space of all periodic vector fields of Sobolev class $H^1$ along
$\gamma$, given by
\begin{equation}\label{eq:defindexform}
I_\gamma(V,W)=\int_0^1g\big(\Ddt V,\Ddt W)+g(RV,W)\,\mathrm dt,
\end{equation}
where we set $R=R(\dot\gamma,\cdot)\dot\gamma$.
Closed geodesics in $M$ are the critical points of the geodesic action functional
$f(\gamma)=\frac12\int_0^1g(\dot\gamma,\dot\gamma)\,\mathrm dt$ defined in
the \emph{free loop space} $\Omega M$ of $M$; $\Omega M$ is the Hilbert manifold of
all closed curves in $M$ of Sobolev class $H^1$. The index form $I_\gamma$ is the second
variation of $f$ at the critical point $\gamma$; unless $g$ is positive definite,
the Morse index of $f$ at each nonconstant critical point is infinite.
The notion of Morse index is replaced by the notion of spectral flow.

\subsection{Periodic frames and trivializations}
Consider a smooth periodic orthonormal frame $\mathbf T$ along $\gamma$, i.e.,
a smooth family $[0,1]\ni t\mapsto T_t$ of isomorphisms:
\begin{equation}
\label{eq:defperiodicframe}
T_t:\R^n\longrightarrow T_{\gamma(t)}M,
\end{equation}
 with $T_0=T_1$, and \begin{equation}\label{eq:deltaij}g(T_te_i,T_te_j)=\epsilon_i\delta_{ij},\end{equation}
 where $\{e_i\}_{i=1,\ldots,n}$ is the canonical basis of $\R^n$, $\epsilon_i\in\{-1,1\}$ and
 $\delta_{ij}$ is the Kronecker symbol.
 The existence of such a frame is guaranteed by the orientability assumption on the closed
 geodesic.
The pull-back by $T_t$ of the metric $g$ gives a symmetric nondegenerate bilinear form
$G$ on $\R^n$, whose index is the same as the index of $g$; note that this pull-back does not depend
on $t$, by the orthogonality assumption on the frame $\mathbf T$. In the sequel, we will
also denote by $G:\R^n\to\R^n$ the symmetric linear operator defined by $(Gv)\cdot w$;
By \eqref{eq:deltaij}, $G$ satisfies
\begin{equation}\label{eq:Gsymmetry}
G^2=I.
\end{equation}

 For all $t\in\left]0,1\right]$, define by $\mathcal H^\gamma_t$ the Hilbert space
 of all $H^1$-vector fields $V$ along $\gamma\vert_{[0,t]}$ satisfying
 \[
 T_0^{-1}V(0)=T_t^{-1}V(t).
 \]

Observe that the definition of $\mathcal H^\gamma_t$ depends on the choice of the periodic frame $\mathbf T$, however,
 $\mathcal H^\gamma_1$, which is the space of all periodic vector fields along $\gamma$, does not depend
 on $\mathbf T$. Although in principle there is no necessity of fixing a specific Hilbert space
 inner product, it will be useful to have one at disposal, and this will be chosen as follows.
 For all $t\in\left]0,1\right]$, consider the Hilbert space
 \[
 H^1_{\mathrm{per}}\big([0,t],\R^n\big)=\Big\{\overline V\in H^1\big([0,t],\R^n):\overline V(0)=\overline V(t)\Big\}.
 \]
 There is a natural Hilbert space inner product in $H^1_{\mathrm{per}}\big([0,t],\R^n\big)$ given by
 \begin{equation}\label{eq:innprodhi1per}
 \langle\overline V,\overline W\rangle=\overline V(0)\cdot\overline W(0)+\int_0^t\overline V'(s)\cdot\overline W'(s)\,\mathrm ds,
 \end{equation}
 where $\cdot$ is the Euclidean inner product in $\R^n$.
 The map $\Psi_t:\mathcal H^\gamma_t\to H^1_{\mathrm{per}}\big([0,t],\R^n\big)$ defined by $\Psi_t(V)=\overline V$,
 where $\overline V(s)=T_s^{-1}(V(s))$ is an isomorphism; the space $\mathcal H^\gamma_t$ will be endowed with
 the pull-back of the inner product \eqref{eq:innprodhi1per} by the isomorphism $\Psi_t$.
 Denote by $\overline R_t\in\mathrm{L}(\R^n)$ the pull-back by $T_t$ of the endomorphism $R_{\gamma(t)}=R(\dot\gamma,\cdot)\dot\gamma$
 of $T_{\gamma(t)}M$:
 \[
 \overline R_t=T_t^{-1}\circ R_{\gamma(t)}\circ T_t;
 \]
 observe that $t\mapsto\overline R_t$ is a smooth map of $G$-symmetric endomorphisms of $\R^n$. Finally, denote by $\Gamma_t\in\mathrm{L}(\R^n)$ the
 \emph{Christoffel symbol} of the frame $\mathbf T$, defined by
 \[
 \Gamma_t(v)=T_t^{-1}\big(\Ddt V\big)-\frac{\mathrm d}{\mathrm dt}\overline V(t),
 \]
 where $\overline V$ is any vector field satisfying $\overline V(t)=v$, and $V=\Psi^{-1}_t(\overline V)$.
 The push-forward  by $\Psi_t$ of the index form $I_\gamma$ on $\mathcal H^\gamma_t$ is given by the bounded
 symmetric bilinear form $\overline I_t$ on $H^1_{\mathrm{per}}\big([0,t],\R^n\big)$ defined by
 \begin{multline}\label{eq:defoverlineIt}
 \overline I_t(\overline V,\overline W)=\int_0^tG\big(\overline V'(s),\overline W'(s)\big)+G\big(\Gamma_s\overline V(s),\overline W'(s)\big)+
 G\big(\Gamma_s\overline W(s),\overline V'(s)\big)\\+G\big(\Gamma_s\overline V(s),\Gamma_s\overline W(s)\big)+G\big(\overline R_s\overline V(s),\overline W(s)\big)\,\mathrm ds.
 \end{multline}
 Finally, for $t\in\left]0,1\right]$, we will consider the isomorphism \[\Phi_t:H^1_{\mathrm{per}}\big([0,t],\R^n\big)\to H^1_{\mathrm{per}}\big([0,1],\R^n\big),\]
 defined by $\overline V\mapsto\widetilde V$, where $\widetilde V(s)=\overline V(st)$, $s\in[0,1]$.
 The push-forward by $\Phi_t$ of the bilinear form $\overline I_t$ is given by the bounded symmetric
 bilinear form $\widetilde I_t$ on $H^1_{\mathrm{per}}\big([0,1],\R^n\big)$ defined by:
 \begin{multline}\label{eq:deftildeIt}
 \widetilde I_t(\widetilde V,\widetilde W)=\frac1{t^2}\int_0^1G\big(\widetilde V'(r),\widetilde W'(r)\big)+tG\big(\Gamma_{tr}\widetilde V(r),\widetilde W'(r)\big)+
t G\big(\Gamma_{tr}\widetilde W(r),\widetilde V'(r)\big)\\+t^2G\big(\Gamma_{tr}\widetilde V(r),\Gamma_{tr}\widetilde W(r)\big)+t^2G\big(\widetilde R_{tr}\widetilde V(r),\widetilde W(r)\big)\,\mathrm dr.
  \end{multline}

\subsection{Spectral flow of a periodic geodesic}
For $t\in\left]0,1\right]$, define the Fredholm bilinear form $B_t$ on the Hilbert space $H^1_{\mathrm{per}}\big([0,1],\R^n\big)$
by setting
\begin{equation}\label{eq:defformaBt}
B_t=t^2\cdot\widetilde I_t.
\end{equation}
{}From \eqref{eq:deftildeIt} we obtain immediately the following result.
\begin{lem}
The map $\left]0,1\right]\ni t\mapsto B_t$ can be extended continuously to $t=0$ by setting:
\[B_0(\widetilde V,\widetilde W)=\int_0^1G\big(\widetilde V'(r),\widetilde W'(r)\big)\,\mathrm dr.\eqno{\qed}\]
\end{lem}
Observe that $\Ker B_0$ is one-dimensional, and it consists of all constant vector fields.
\begin{prop}
For all $t\in\left[0,1\right]$, the bilinear form $\widetilde I_t$ on $H^1_{\mathrm{per}}\big([0,1],\R^n\big)$ is
represented with respect to the inner product \eqref{eq:innprodhi1per} by a compact perturbation of the
symmetry $\mathfrak J$ of $H^1_{\mathrm{per}}\big([0,1],\R^n\big)$ given by $\widetilde V\mapsto G\widetilde V$.
\end{prop}
\begin{proof}
First, observe that $B_t$ is a compact perturbation of $B_0$. Namely, from \eqref{eq:deftildeIt} we get:
\begin{multline*}
B_t(\widetilde V,\widetilde W)-B_0(\widetilde V,\widetilde W)=\int_0^1tG\big(\Gamma_{tr}\widetilde V(r),\widetilde W'(r)\big)+
t G\big(\Gamma_{tr}\widetilde W(r),\widetilde V'(r)\big)\\+t^2G\big(\Gamma_{tr}\widetilde V(r),\Gamma_{tr}\widetilde W(r)\big)+t^2G\big(\widetilde R_{tr}\widetilde V(r),\widetilde W(r)\big)\,\mathrm dr.
\end{multline*}
The integral above defines a bilinear map which is continuous in the $H^{\frac12}$-topology, and thus it is represented by a compact
operator, since the inclusion $H^1\hookrightarrow H^{\frac12}$ is compact.
Next, observe that $B_0$ is represented by a compact perturbation of the symmetry $\mathfrak J$. For,
\[\langle\mathfrak J\widetilde V,\widetilde W\rangle-B_0(\widetilde V,\widetilde W)=G\big(\widetilde V(0),\widetilde W(0)\big),\]
which is continuous in the $C^0$-topology, hence represented by a compact operator. Note that $\mathfrak J$ is self-adjoint and,
by \eqref{eq:Gsymmetry}, $\mathfrak J^2=I$; thus, $\mathfrak J$ is a symmetry.
This concludes the proof.
\end{proof}
\begin{defin}\label{defin:spectral}
The \emph{spectral flow} $\spfl(\gamma)$ of the closed geodesic $\gamma$ is defined as the spectral flow of the continuous path of
Fredholm bilinear forms $[0,1]\ni t\mapsto B_t$ on the Hilbert space $H^1_{\mathrm{per}}\big([0,1],\R^n\big)$.
\end{defin}

\begin{rem}
The fact that the definition of $\spfl(\gamma)$ does not depend on the choice of a smooth periodic orthonormal
frame along $\gamma$ is a nontrivial fact, and it will be proven in next subsection by giving an explicit formula
for its computation.

We observe here that the paths of Fredholm bilinear forms $B_t$ as above produced by two distinct periodic
trivializations of the tangent bundle are in general neither fixed endpoint homotopic, nor cogredient.
Namely, two distinct trivializations \emph{differ} by a closed path in the (connected component of the
identity of the) Lie group $\mathrm O(G)$ of all $G$-preserving linear isomorphisms of $\R^n$,
which is not simply connected.
\end{rem}
\end{section}

\subsection{Computation of the spectral flow}
There is an integer valued invariant associated to every (fixed endpoints) geodesic
in a semi-Riemannian manifold $(M,g)$, called the \emph{Maslov index}. This is a symplectic
invariant, which is computed as an intersection number in the Lagrangian Grasmannian
of a symplectic vector space. Details on the definition and the computation of
the Maslov index for a given geodesic $\gamma$, that will be denoted by $\iMaslov(\gamma)$
can be found in \cite{CRASP04, asian, topology}.

As for the definition of spectral flow,
there are several conventions in the literature concerning the computation of the contribution
to the Maslov index of the endpoints of the geodesic. In this section we will convention\footnote{%
This is not a standard choice in the literature.}
that in the computation of the Maslov index $\iMaslov(\gamma)$ it is also considered the
contribution of the initial point of $\gamma$; the value of this contribution is easily computed to be
equal to $\mathrm n_-(g)$, which is the index of the semi-Riemannian metric tensor $g$.

Recall that a \emph{Jacobi field} along $\gamma$ is a smooth vector field $J$ along $\gamma$ that satisfies the second order linear equation
\[\
DDdt J(t)=R\big(\dot\gamma(t),J(t)\big)\,\dot\gamma(t),\quad t\in[0,1].
\]

Let us denote by $\mathcal J_\gamma$ the $2n$-dimensional real vector space of all Jacobi fields along $\gamma$.
Let us introduce the following spaces:
\begin{eqnarray*}
&&\mathcal J_\gamma^{\text{per}}=\Big\{J\in\mathcal J_\gamma:J(0)=J(1),\ \Ddt J(0)=\Ddt J(1)\Big\},\\
&&\mathcal J_\gamma^0=\Big\{J\in\mathcal J_\gamma:J(0)=J(1)=0\Big\},\quad\text{and}\\ &&\mathcal J_\gamma^\star=\Big\{J\in\mathcal J_\gamma:J(0)=J(1)\Big\}.
\end{eqnarray*}
It is well known that $\mathcal J_\gamma^{\text{per}}$ is the kernel of the index form $I_\gamma$ defined in \eqref{eq:defindexform},
while $\mathcal J_\gamma^0$ is the kernel of the restriction of the index form to the space of vector fields
along $\gamma$ vanishing at the endpoints. We denote by $\mathrm n_\gamma^{\text{per}}$ and $\mathrm n_\gamma^0$ the
dimensions of $\mathcal J_\gamma^{\text{per}}$ and $\mathcal J_\gamma^0$ respectively. The nonnegative integer
$\mathrm n_\gamma^{\text{per}}$ is the nullity of $\gamma$ as a periodic geodesic, i.e., the nullity of the Hessian
of the geodesic action functional at $\gamma$ in the space of closed curves. Observe that $\mathrm n_\gamma^{\text{per}}\ge1$,
as $\mathcal J_\gamma^{\text{per}}$ contains the one-dimensional space spanned by the tangent field $J=\dot\gamma$.
Similarly, $\mathrm n_\gamma^0$ is the nullity of $\gamma$ as a fixed endpoint geodesic, i.e., it is the nullity of
the Hessian of the geodesic action functional at $\gamma$ in the space of fixed endpoints curves in $M$.
In this case, $\mathrm n_\gamma^0>0$ if and only if $\gamma(1)$ is conjugate to $\gamma(0)$ along $\gamma$.

Given a semi-Riemannian geodesic $\gamma$,
the spectral flow of the path of symmetric Fredholm bilinear forms $[0,1]\ni t\mapsto B_t$ restricted
to the space $H^1_0\big([0,1],\R^n\big)$ will be denoted by $\spfl_0(\gamma)$. A formula giving
the value of this integer is proven in \cite[Proposition~2]{CRASP04}:
\begin{prop}
Given any (closed) semi-Riemannian geodesic $\gamma$, the following equality holds:
\begin{equation}\label{eq:spfl0}
\spfl_0(\gamma)=\mathrm n^0_\gamma-\mathrm n_-(g)-\iMaslov(\gamma).{\qed}
\end{equation}
\end{prop}
Finally, the last ingredient needed for the computation of the spectral flow of a closed geodesic
is the so called \emph{index of concavity} of $\gamma$, that will be denoted by
$\mathrm i_{\text{conc}}(\gamma)$. This is a nonnegative integer invariant associated to
periodic solutions of Hamiltonian systems, first introduced by M. Morse \cite{Morseconc} in the context
of closed Riemannian geodesic. In our notations, $\mathrm i_{\text{conc}}(\gamma)$ is equal to
the index of the symmetric bilinear form:
\[
(J_1,J_2)\longmapsto g\big(\Ddt J_1(1)-\Ddt J_1(0),J_2(0)\big),
\]
defined on the vector space $\mathcal J_\gamma^\star$. It is not hard to show
that this bilinear form is symmetric, in fact, it is given by the restriction of
the index form $I_\gamma$ to $\mathcal J_\gamma^\star$.

It is now easy to apply Theorem~\ref{riduzionespettro} in order to obtain a formula for
the spectral flow of an oriented closed geodesic.
\bth{Morseperiodic}
Let $(M,g)$ be a semi-Riemannian manifold and let $\gamma:[0,1]\to M$ be a closed oriented geodesic in $M$.
Then, the spectral flow $\spfl(\gamma)$ is given by the following formula:
\begin{equation}\label{eq:formulaspectralflowperiodic}
\spfl(\gamma)=\Dim\big(\mathcal J_\gamma^{\text{per}}\cap\mathcal J_\gamma^0\big)-\iMaslov(\gamma)-\mathrm i_{\text{conc}}(\gamma)-\mathrm n_-(g).
\end{equation}
\eth
\begin{proof}
Set $\mathcal H=H^1_{\mathrm{per}}\big([0,1],\R^n\big)$, $\mathcal V=H^1_{0}\big([0,1],\R^n\big)$ in Theorem~\ref{riduzionespettro};
The difference $\spfl(\gamma)-\spfl_0(\gamma)$ is thus given by the  sum
of six terms, that are computed easily as follows.
The space $\mathcal V^{\perp_{B_0}}$ coincides with the kernel of $B_0$, and it is given by the one dimensional space
of constant vector fields on $[0,1]$; the restriction of $B_0$ to such space vanishes identically.
Moreover, $\mathcal V\cap\mathcal V^{\perp_{B_0}}=\mathcal V\cap\Ker B_0=\{0\}$.
A straightforward partial integration argument shows that the space $\mathcal V^{\perp_{B_1}}$
is given by $\mathcal J_\gamma^\star$ By definition, the index of the restriction of $B_1$
to this space equals $\mathrm i_{\text{conc}}(\gamma)$.
The space $\mathcal V\cap\mathcal V^{\perp_{B_1}}=\Ker\big(B_1\big\vert_{\mathcal V\times\mathcal V}\big)$
is given by $\mathcal J_\gamma^0$.
Finally, $\Ker B_1 =\mathcal J_\gamma^{\text{per}}$, thus $\Ker B_1 \cap\mathcal V=\mathcal J^{\text{per}}_\gamma\cap
\mathcal J_\gamma^0$. Formula \eqref{eq:formulaspectralflowperiodic} follows
now immediately from \eqref{eq:spfl0}.
\end{proof}

Formula \eqref{eq:formulaspectralflowperiodic} proves in particular that the definition of spectral flow
for a periodic geodesic $\gamma$ does not depend on the choice of an orthonormal frame along $\gamma$.

\begin{rem}
Our definition of spectral flow along a closed geodesic has used a periodic orthonormal frame
along the geodesic, which exists only if the geodesic is orientation preserving.  We observe however
that the right hand side of formula \eqref{eq:formulaspectralflowperiodic} is defined for every
closed geodesic, regardless of its orientability, which suggests that \eqref{eq:formulaspectralflowperiodic}
can be taken as the definition of spectral flow in the nonorientable case.
Let us sketch briefly how the right-hand side of \eqref{eq:formulaspectralflowperiodic} can be obtained
as a spectral flow of paths of Fredholm operators. Given a nonorientable closed geodesic
$\gamma$, choose an arbitrary smooth frame $\mathbf T$ along $\gamma$ as in \eqref{eq:defperiodicframe},
which will \emph{not} satisfy $T_0=T_1$; set $S=T_1^{-1}T_0\in\mathrm{GL}(\R^n)$.
Then, the spectral flow $\spfl(\gamma)$ is defined as the difference $\mathfrak{sf}_S(\gamma)-\mathrm n_S$, where
$\mathfrak{sf}_S(\gamma)$ is the spectral flow of the path of Fredholm bilinear forms
$[0,1]\ni t\mapsto B_t$ given in \eqref{eq:defformaBt} on the space
\[
H^1_S\big([0,1],\R^n\big)=\Big\{\widetilde V\in H^1\big([0,1],\R^n\big):\widetilde V(1)=S\widetilde V(0)\Big\},
\]
and $\mathrm n_S$ is the index of the restriction of the metric tensor $g$ to the image of the operator $S-I$
(compare with Definition~\ref{defin:spectral}).  Note that $S=I$ in the orientation preserving case.
With such definition, formula \eqref{eq:formulaspectralflowperiodic}
 holds also in the nonorientation preserving case. This is proven easily using Theorem~\ref{riduzionespettro},
 as in the proof of Theorem~\ref{Morseperiodic}. One sets $\mathcal H=H^1_{S}\big([0,1],\R^n\big)$,
 $\mathcal V=H^1_{0}\big([0,1],\R^n\big)$, and observes that in this case the  space   $\mathcal V^{\perp_{B_0}}$
 consists of all affine maps $\widetilde V:[0,1]\to\C^n$ of the form $\widetilde V(t)=(S-I)B+B$, where
 $B$ is an arbitrary vector in $\C^n$. The restriction of the the Hermitian form $B_0$ to such space
 equals the index of the restriction of $g$ to the image of $S-I$, from which the desired conclusion
 follows.
\end{rem}


\appendix
\begin{section}{Group actions and fibrations over the infinite dimensional Grassmannian}
\label{app:grassmannians}
In this appendix we will study the fibrations over the Grassmannian of all closed subspaces of a Hilbert space $\mathcal H$
determined by the actions of the general linear group $\mathrm{GL}(\mathcal H)$ and of the orthogonal group $\mathrm O(\mathcal H)$.

Let $\mathcal H$ be an infinite dimensional separable Hilbert space; denote, as in the previous sections, by $\mathrm L(\mathcal H)$ the Banach algebra of all bounded linear operators on $\mathcal H$, by  $\mathrm L_{\mathrm{sa}}(\mathcal H)$ (resp., $\mathrm L_{\mathrm{as}}(\mathcal H)$)
the subspace of $\mathrm L(\mathcal H)$  of self-adjoint operators (resp., of anti-symmetric operators),  by $\mathrm{GL}(\mathcal H)$ the
Banach Lie group of all bounded linear isomorphisms of $\mathcal H$ and by $\mathrm O(H)$ the subset of
$\mathrm{GL}(\mathcal H)$ consisting of isometries of $\mathcal H$:
\[
\mathrm O(\mathcal H)=\big\{T\in\mathrm{GL}(\mathcal H):T^*T=TT^*=I\big\}.
\]

By a well known result due to Kuiper \cite{Kui}, $\mathrm O(\mathcal H)$ is contractible; $\mathrm O(\mathcal H)$ is a smooth
embedded submanifold of $\mathrm{GL}(H)$, being the inverse image $f^{-1}(I)\cap\mathrm{GL}(\mathcal H)$
of the submersion $\mathrm L(\mathcal H)\ni T\mapsto T^*T\in\mathrm L_{\mathrm{sa}}(\mathcal H)$.
The tangent space $T_1\mathrm{GL}(\mathcal H)$ is $\mathrm L(\mathcal H)$; the tangent space
$T_1\mathrm O(\mathcal H)$ is the subspace $\mathrm L_{\mathrm{as}}(\mathcal H)$.
Denote by $\mathrm{Gr}(\mathcal H)$ the Grassmannian of all closed subspaces of $\mathcal H$, which is
a metric space endowed with the metric $\mathrm{dist}(\mathcal V,\mathcal W)=\Vert P_{\mathcal V}-P_{\mathcal W}\Vert$.
There is an action $\mathrm{GL}(\mathcal H)\times\mathrm{Gr}(\mathcal H)\to\mathrm{Gr}(\mathcal H)$ given by $(T,\mathcal V)\mapsto T(\mathcal V)$.

The set $\mathrm{Gr}(\mathcal H)$ has a real analytic Banach manifold structure, the action of $\mathrm{GL}(\mathcal H)$ is analytic,
and so is its restriction to the orthogonal group (see for instance \cite{AbbMej3}). The connected components of $\mathrm{Gr}(\mathcal H)$ are
the sets
\[
\mathrm{Gr}_{k_1,k_2}(\mathcal H)=\big\{\mathcal V\in\mathrm{Gr}(\mathcal H):\mathrm{dim}(\mathcal V)=k_1,\ \mathrm{dim}(\mathcal V^\perp)=k_2\big\},
\]
where $k_1,k_2\in\mathds N\cup\{+\infty\}$ are not both finite numbers. The action of $\mathrm O(\mathcal H)$ is transitive on each connected
component of $\mathrm{Gr}(\mathcal H)$. For all $\mathcal W\in\mathrm{Gr}(\mathcal H)$, the tangent space
$T_{\mathcal W}\mathrm{Gr}(\mathcal H)$ is identified with the Banach space $\mathrm  L(\mathcal W,\mathcal W^\perp)$
of all bounded linear operators $X:\mathcal W\to\mathcal W^\perp$.

Here comes a simple result on group actions, submersion and fibrations.
\begin{lem}\label{thm:actsubmfibr}
Let $M$ be a Banach manifold and let  $G$ be a Banach Lie group acting smoothly and transitively on $M$: \[G\times M\ni (g,m)\mapsto g\cdot m\in M.\]
Let $m\in M$ be fixed, and denote by $\beta_m:G\to M$ the map $\beta_m(g)=g\cdot m$.
\begin{itemize}
\item[(a)] If $\beta_m$ is a submersion at $g=1$, then $\beta_m$ is a submersion.
\smallskip

\item[(b)] If $\beta_m$ is a submersion, then $\beta_m$ is a smooth fibration with typical fiber the isotropy group $G_m$.
\end{itemize}
\end{lem}

\begin{proof}
Denote by $L_g:G\to G$ the left translation by $g$: $L_g(h)=gh$, and by $\gamma_g:M\to M$ the diffeomorphism $\gamma_g(m)=g\cdot m$.
Then, $\beta_m\circ L_g=\gamma_g\circ\beta_m$; differentiating at $h=1$ gives
\[
\mathrm d\beta_m(g)\circ\mathrm dL_g(1)=\mathrm d\gamma_g(m)\circ\mathrm d\beta_m(1).
\]
Note that $\mathrm dL_g(1)$ and $\mathrm d\gamma_g(m)$ are isomorphisms. Thus, if $\mathrm d\beta_m(1)$ is surjective,
then so is $\mathrm d\beta_m(g)$. Similarly, if $\Ker\big(\mathrm d\beta_m(1)\big)$ is complemented,
then so is $\Ker\big(\mathrm d\beta_m(g)\big)=\mathrm dL_1(1)\big[\Ker\big(\mathrm d\beta_m(1)\big)\big]$. This proves part (a).

For part (b),  it suffices to show the existence of local trivializations. Note that the stabilizer $G_m$ of $m$ is a Lie subgroup of $G$, being the inverse
image of a value of a submersion: $G_m=\beta_m^{-1}(m)$. Let $S:U\subseteq M\to G$ be a local section of $\beta_m$; local sections exists
by the assumption that $\beta_m$ is a submersion. Then, a trivialization of $\beta_m^{-1}(U)$ is given by
\[
U\times G_m\ni(x,g)\longmapsto s(x)g\in\beta^{-1}(U).
\]
Obviously, this map is smooth, and its inverse is given by \[\beta_m^{-1}(U)\ni h\longmapsto\big(h\cdot m,s(h\cdot m)^{-1}h\big)\in U\times G_m,\]
which is also smooth.
\end{proof}

\begin{prop}
Let $\mathcal W\in\mathrm{Gr}(\mathcal H)$ be fixed and let $\mathrm{Gr}_{k_1,k_2}(\mathcal H)$ be its connected component
in $\mathrm{Gr}(\mathcal H)$. The map $\beta_{\mathcal W}:\mathrm{GL}(\mathcal H)\to\mathrm{Gr}_{k_1,k_2}(\mathcal H)$,
defined by $\beta_m(T)=T(\mathcal W)$, is a real analytic fibration. The same conclusion holds for the restriction of
$\beta_m$ to $\mathrm O(\mathcal H)$.
\end{prop}

\begin{proof}
By part (a) and (b) of Lemma~\ref{thm:actsubmfibr}, it suffices to show that the linear map $\mathrm d\beta_{\mathcal W}(1):\mathrm L(\mathcal H)\to
\mathrm L(\mathcal W,\mathcal W^\perp)$ is surjective and that it has complemented kernel, as well as its restriction to $\mathrm L_{\mathrm{as}}(\mathcal H)$.
An explicit computation gives:
\[\mathrm d\beta_{\mathcal W}(1)X=P_{\mathcal W^\perp}\circ X\vert_{\mathcal W},\quad\forall\,X\in\mathrm L(\mathcal H),\]
where $P_{\mathcal W^\perp}$ is the orthogonal projection onto $\mathcal W^\perp$. Writing $X:\mathcal W\oplus\mathcal W^\perp\to\mathcal W\oplus\mathcal W^\perp$
in block form:
\[X=\begin{pmatrix}X_{11}&X_{12}\\ X_{21}&X_{22}\end{pmatrix},\]
then $\mathrm d\beta_{\mathcal W}(1)X=X_{21}:\mathcal W\to\mathcal W^\perp$. Clearly, a complement in $\mathrm L(\mathcal H)$ for the kernel
of this map is the closed subspace of $\mathrm L(\mathcal H)$ consisting of operators $Y$ that are written in block form as $Y=\begin{pmatrix}0&0\\Y_{21}&0\end{pmatrix}$,
where $Y_{21}\in\mathrm L(\mathcal W,\mathcal W^\perp)$.

Similarly, the kernel of  $\mathrm d\beta_{\mathcal W}(1):\mathrm L_{\mathrm{as}}(\mathcal H)\to
\mathrm L(\mathcal W,\mathcal W^\perp)$ consists of all anti-symmetric operators $X$ that are written in block form
as $X=\begin{pmatrix}X_{11}&0\\0&X_{22}\end{pmatrix}$, where $X_{11}\in\mathrm L_{\mathrm{as}}(\mathcal W)$ and $X_{22}\in\mathrm L_{\mathrm{as}}(\mathcal W^\perp)$.
A complement for this space in $\mathrm L_{\mathrm{as}}(\mathcal H)$ is given by the closed subspace of $\mathrm L_{\mathrm{as}}(\mathcal H)$ consisting
of all operators $Y$ that have block form $Y=\begin{pmatrix}0&Y_{12}\\Y_{12}^*&0\end{pmatrix}$, with $Y_{12}\in\mathrm L(\mathcal W^\perp,\mathcal W)$.

Moreover, it is easy to check that $\mathrm d\beta_{\mathcal W}(1):\mathrm L_{\mathrm{as}}(\mathcal H)\to\mathrm L(\mathcal W,\mathcal W^\perp)$
(and thus also $\mathrm d\beta_{\mathcal W}(1):\mathrm L(\mathcal H)\to\mathrm L(\mathcal W,\mathcal W^\perp)$) is surjective. Namely, given
any $A\in\mathrm L(\mathcal W,\mathcal W^\perp)$, there exists $X\in\mathrm L_{\mathrm{as}}(\mathcal H)$ whose lower down block $X_{21}$
relative to the decomposition $\mathcal H=\mathcal W\oplus\mathcal W^\perp$ equals $A$, for instance, $X=\begin{pmatrix}0&-A^*\\A&0\end{pmatrix}$.
This concludes the proof.
\end{proof}

\begin{cor}\label{thm:esistenzatrivializzazioniortogonali}
Given any curve $\mathcal V:[a,b]\to\mathrm{Gr}(\mathcal H)$ of class $C^k$, $k=0,\ldots,\infty,\omega$, given any $\mathcal W$ in the connected
component $\mathrm{Gr}_{k_1,k_2}(\mathcal H)$ of $\mathcal V_a$ in $\mathrm{Gr}(\mathcal H)$ and any isometry $\varphi:\mathcal H\to\mathcal H$
such that $\varphi(\mathcal W)=\mathcal V_a$, then there exists a curve $\Phi:[a,b]\to\mathrm O(\mathcal H)$ of
class $C^k$ such that $\Phi_t(\mathcal W)=\mathcal V_t$ for all $t\in[a,b]$ and with $\Phi_a=\varphi$.
\end{cor}

\begin{proof}
$\Phi$ is a lifting of the curve $\mathcal V$ in the fibration $\beta_{\mathcal W}$:
\[\vcenter{\xymatrix{&&\mathrm O(\mathcal H)\ar[d]^{\beta_{\mathcal W}}\cr[a,b]\ar[rr]_{\mathcal V}\ar@{.>}[urr]^\Phi &&\mathrm{Gr}_{k_1,k_2}(\mathcal H)}}\qedhere\]
\end{proof}

\medskip

It is interesting to restate the result above in terms of symmetries. Recall that by a symmetry of $\mathcal H$ we mean
a self-adjoint operator $\mathfrak I$ on $\mathcal H$ such that $\mathfrak I^2=I$. Denote by $\mathfrak S(\mathcal H)$ the closed subset
of $\mathrm O(\mathcal H)$ consisting of all symmetries of $\mathcal H$; the bijection $\mathfrak S(\mathcal H)\ni\mathfrak I\mapsto\Ker(\mathfrak I-I)\in\mathrm{Gr}(\mathcal H)$
is a homeomorphism, whose inverse is
\[
\mathrm{Gr}(\mathcal H)\ni\mathcal V\mapsto P_{\mathcal V}-P_{\mathcal V^\perp}\in\mathfrak S(\mathcal H).
\]
This
bijection carries the action $\mathrm O(\mathcal H)\times\mathrm{Gr}(\mathcal H)\ni (U,\mathcal V)\mapsto U(\mathcal V)\in\mathrm{Gr}(\mathcal H)$
into the cogredient action:
\[\mathrm O(\mathcal H)\times\mathfrak S(\mathcal H)\ni (U,\mathfrak I)\longmapsto U\mathfrak IU^*\in\mathfrak S(\mathcal H),\]
i.e., if $\mathfrak I=P_{\mathcal V}-P_{\mathcal V^\perp}$, then $U\mathfrak IU^*=P_{U(\mathcal V)}-P_{U(\mathcal V)^\perp}$.
Thus, Corollary~\ref{thm:esistenzatrivializzazioniortogonali} can be translated as follows.
\begin{cor}\label{thm:esistenzatrivializzazioniortogonalisym}
Let $[a,b]\ni t\mapsto\mathfrak I_t\in\mathrm{GL}(\mathcal H)$ be a map of class $C^k$, $k=0,\ldots,\infty,\omega$, where
$\mathfrak I_t\in\mathfrak S(\mathcal H)$ for all $t$. Then, there exists a $C^k$ map $[a,b]\ni t\mapsto U_t\in\mathrm O(\mathcal H)$
and a fixed symmetry $\mathfrak I\in\mathfrak S(\mathcal H)$ such that $U_t\mathfrak I_tU_t^*=\mathfrak I$ for all $t$.\qed
\end{cor}
\end{section}

\end{document}